\documentclass[11pt]{article}

\usepackage{amssymb, amsmath, amsthm, graphicx}
\usepackage[left=1in,top=1in,right=1in]{geometry}

\usepackage{subfigure}

      \theoremstyle{plain}
      \newtheorem{theorem}{Theorem}[section]
      \newtheorem{lemma}[theorem]{Lemma}
            
      \newtheorem{corollary}[theorem]{Corollary}
            \newtheorem{observation}[theorem]{Observation}

      \theoremstyle{definition}

      \theoremstyle{remark}

\title{Coloring intersection graphs of $x$-monotone curves in the plane
}

\author{Andrew Suk\thanks{Massachusetts Institute of Technology, Cambridge. Supported by an NSF Postdoctoral Fellowship. Email: {\tt asuk@math.mit.edu}.  }
}

\begin{document}

\maketitle

\medskip

\begin{abstract}

A class of graphs $\mathcal{G}$ is $\chi$-\emph{bounded} if the chromatic number of the graphs in $\mathcal{G}$ is bounded by some function of their clique number.  We show that the class of intersection graphs of simple families of $x$-monotone curves in the plane intersecting a vertical line is $\chi$-bounded.  As a corollary, we show that the class of intersection graphs of rays in the plane is $\chi$-bounded, and the class of intersection graphs of unit segments in the plane is $\chi$-bounded.

\end{abstract}

\section{Introduction}

For a graph $G$, The \emph{chromatic number} of $G$, denoted by $\chi(G)$, is the minimum number of colors required to color the vertices of $G$ such that any two adjacent vertices have distinct colors.  The \emph{clique number} of $G$, denoted by $\omega(G)$, is the size of the largest clique in $G$.  We say that a class of graphs $\mathcal{G}$ is $\chi$-\emph{bounded} if there exists a function $f:\mathbb{N} \mapsto \mathbb{N}$ such that every $G\in \mathcal{G}$ satisfies $\chi(G) \leq f(\omega(G))$.  Although there are triangle-free graphs with arbitrarily large chromatic number \cite{erdos,west}, it has been shown that certain graph classes arising from geometry are $\chi$-bounded.

Given a collection of objects $\mathcal{F}$ in the plane, the \emph{intersection graph} $G(\mathcal{F})$ has vertex set $\mathcal{F}$ and two objects are adjacent if and only if they have a nonempty intersection.  For simplicity, we will shorten $\chi(G(\mathcal{F})) = \chi(\mathcal{F})$ and $\omega(G(\mathcal{F})) = \omega(\mathcal{F})$.  The study of the chromatic number of intersection graphs of objects in the plane was stimulated by the seminal papers of Asplund and Gr\"unbaum \cite{ap} and Gy\'arf\'as and Lehel \cite{gyarfas,gyarfas2}.  Asplund and Gr\"unbaum showed that if $\mathcal{F}$ is a family of axis parallel rectangles in the plane, then $\chi(\mathcal{F}) \leq 4\omega(\mathcal{F})^2$. Gy\'arf\'as and Lehel \cite{gyarfas,gyarfas2} showed that if $\mathcal{F}$ is a family of chords in a circle, then $\chi(\mathcal{F})\leq 2^{\omega(\mathcal{F})}\omega(\mathcal{F})^3$.  Over the past 50 years, this topic has received a large amount of attention due to its application in VLSI design \cite{vlsi}, map labeling \cite{map}, graph drawing \cite{fox,suk}, and elsewhere.  For more results on the chromatic number of intersection graphs of objects in the plane and in higher dimensions, see \cite{burling,fox,gyarfas2,kim,kostochka,kos1,kos2,sean,sean2,seanL}.

In this paper, we study the chromatic number of intersection graphs of $x$-monotone curves in the plane.  We say a family of curves $\mathcal{F}$ is \emph{simple}, if every pair of curves intersect at most once and no two curves are tangent, that is, if two curves have a common interior point, they must properly cross at that point.  Our main theorem is the following.

\begin{theorem}
\label{main}
The class of intersection graphs of simple families of $x$-monotone curves in the plane intersecting a vertical line is $\chi$-bounded.
\end{theorem}

\noindent  McGuinness proved a similar statement in \cite{sean,sean2}, by showing that if $\mathcal{F}$ is a family of curves in the plane with no three pairwise intersecting members, with the additional property that there exists another curve $\gamma$ that intersects all of the members in $\mathcal{F}$ exactly once, then $\chi(F) < 2^{100}$.  Surprisingly, Theorem \ref{main} does not hold if one drops the condition that the curves intersect a vertical line.  Indeed, a recent result of Pawlik, Kozik, Krawczyk, Larso\'n Micek, Trotter, and Walczak \cite{poles} shows that the class of intersection graphs of segments in the plane is not $\chi$-bounded.  However, it is not clear if the simplicity condition in Theorem \ref{main} is necessary.  As an immediate corollary of Theorem \ref{main}, we have the following.

\begin{corollary}
The class of intersection graphs of rays in the plane is $\chi$-bounded.
\end{corollary}

\noindent By applying partitioning \cite{furedi} and divide and conquer \cite{pach} arguments, Theorem \ref{main} also implies the following results.  Since these arguments are fairly standard, we omit their proofs.

\begin{theorem}
Let $\mathcal{S}$ be a family of segments in the plane, such that no $k$ members pairwise cross.  If the ratio of the longest segment to the shortest segment is bounded by $r$, then $\chi(\mathcal{S}) \leq c_{k,r}$ where $c_{k,r}$ depends only on $k$ and $r$.

\end{theorem}

\begin{theorem}
\label{log}
Let $\mathcal{F}$ be a simple family of $n$ $x$-monotone curves in the plane, such that no $k$ members pairwise cross.  Then $\chi(\mathcal{F}) \leq c_k\log n$, where $c_k$ is a constant that depends only on $k$.
\end{theorem}

\noindent This improves the previous known bound of $(\log n)^{15\log k}$ due to Fox and Pach \cite{fox}.  We note that the Fox and Pach bound holds without the simplicity condition.

 Recall that a \emph{topological graph} is a graph drawn in the plane such that its vertices are represented by points and its edges are represented by non-self-intersecting arcs connecting the corresponding points.  A topological graph is \emph{simple} if every pair of its edges intersect at most once.  Since every $n$-vertex planar graph has at most $3n-6$ edges, Theorem \ref{main} gives a new proof of the following result due to Valtr.

\begin{theorem}[\cite{valtr}]
\label{valtr}
Let $G =(V,E)$ be an $n$-vertex simple topological graph with edges drawn as $x$-monotone curves.  If there are no $k$ pairwise crossing edges in $G$, then $|E(G)| \leq c_kn\log n$, where $c_k$ is a constant that depends only on $k$.

\end{theorem}

\noindent We note that I recently showed that Theorem \ref{valtr} holds without the simplicity condition \cite{suk}.

\section{Definitions and notation}

A curve $C$ in the plane is called a \emph{right-flag} (\emph{left-flag}), if one of its endpoints lies on the $y$-axis, and $C$ is contained in the closed right (left) half-plane.  Let $\mathcal{F}$ be a family of $x$-monotone curves in the plane, such that each member intersects the $y$-axis.  Notice that each curve $C \in \mathcal{F}$ can be partitioned into two parts $C = C_1\cup C_2$, where $C_1$ ($C_2$) is a right-flag (left-flag).  By defining $\mathcal{F}_1$ $(\mathcal{F}_2)$ to be the family of $x$-monotone right-flag (left-flag) curves arising from $\mathcal{F}$, we have  $\chi(\mathcal{F}) \leq \chi(\mathcal{F}_1)\chi(\mathcal{F}_2)$. Hence in order to prove Theorem \ref{main}, it suffices to prove the following theorem on simple families of $x$-monotone right-flag curves.

\begin{theorem}
\label{flag}
Let $\mathcal{F}$ be a simple family of $x$-monotone right-flag curves.  If $\chi(\mathcal{F}) > 2^{(5^{k + 1} - 121)/4}$, then $\mathcal{F}$ contains $k$ pairwise crossing members.
\end{theorem}

\begin{corollary}
Let $\mathcal{F}$ be a simple family of $x$-monotone curves, such that each member intersects the $y$-axis.  If $\chi(\mathcal{F}) > 2^{(5^{k + 1} - 121)/2}$, then $\mathcal{F}$ contains $k$ pairwise crossing members.
\end{corollary}

The rest of this paper is devoted to proving Theorem \ref{flag}.  Given a simple family $\mathcal{F} = \{C_1,C_2,...,C_n\}$ of $n$ $x$-monotone right-flag curves, we can assume that no two curves share a point on the $y$-axis and the curves are ordered from bottom to top.  We let $G(\mathcal{F})$ be the intersection graph of $\mathcal{F}$ such that vertex $i \in V(G(\mathcal{F}))$ corresponds to the curve $C_i$.  We can assume that $G(\mathcal{F})$ is connected.  For a given curve $C_i \in \mathcal{F}$, we say that $C_i$ is at \emph{distance} $d$ from $C_1 \in \mathcal{F}$, if the shortest path from vertex 1 to $i$ in $G(\mathcal{F})$ has length $d$.  We call the sequence of curves $C_{i_1},C_{i_2},...,C_{i_p}$ a \emph{path} if the corresponding vertices in $G(\mathcal{F})$ form a path, that is, the curve $C_{i_j}$ intersects $C_{i_{j + 1}}$ for $j = 1,2,...,p-1$.
For $i < j$, we say that $C_i$ \emph{lies below} $C_j$, and $C_j$ \emph{lies above} $C_i$.  We denote $x(C_i)$ to be the $x$-coordinate of the right endpoint of the curve $C_i$.  Given a subset of curves $\mathcal{K}\subset \mathcal{F}$, we denote

$$x(\mathcal{K}) = \min\limits_{C \in \mathcal{K}}(x(C)).$$

For any subset $I \subset \mathbb{R}$, we let $\mathcal{F}(I) = \{C_i\in \mathcal{F}: i \in I\}$.  If $I$ is an interval, we will shorten $\mathcal{F}((i,j))$ to $\mathcal{F}(i,j)$, $\mathcal{F}([i,j])$ to $\mathcal{F}[i,j]$, $\mathcal{F}((i,j])$ to $\mathcal{F}(i,j]$, and $\mathcal{F}([i,j))$ to $\mathcal{F}[i,j)$.

For $\alpha \geq 0$, a finite sequence $\{r_i\}_{i = 0}^m$ of integers is called an $\alpha$\emph{-sequence of $\mathcal{F}$} if for $r_0 = \min\{i: C_i \in \mathcal{F}\}$ and $r_m = \max\{i: C_i\in \mathcal{F}\}$, the subsets $\mathcal{F}[r_0,r_1], \mathcal{F}(r_1,r_2],...,\mathcal{F}(r_{m-1},r_m]$ satisfy

$$\chi(\mathcal{F}[r_0,r_1]) = \chi(\mathcal{F}(r_1,r_2]) = \cdots = \chi(\mathcal{F}(r_{m-2},r_{m-1}]) =  \alpha$$

\noindent and

$$\chi(\mathcal{F}(r_{m-1},r_m]) \leq \alpha.$$

\medskip

\noindent \textbf{Organization.}  In the next section, we will prove several combinatorial lemmas on ordered graphs, that will be used repeatedly throughout the paper.  In Section 4, we will prove several lemmas based on the assumption that Theorem \ref{flag} holds for $k < k'$.  Then in Section 5, we will prove Theorem~\ref{flag} by induction on $k$.

\section{Combinatorial coloring lemmas}

We will make use of the following lemmas.  The first lemma is on ordered graphs $G = ([n],E)$, whose proof can be found in \cite{seanL}.  For sake of completeness, we shall add the proof.  Just as before, for any interval $I\subset \mathbb{R}$, we denote $G(I)\subset G$ to be the subgraph induced by vertices $V(G)\cap I$.

\begin{lemma}
\label{sequence}
Given an ordered graph $G = ([n],E)$, let $a,b\geq 0$ and suppose that $\chi(G) > 2^{a + b + 1}$.  Then there exists an induced subgraph $H \subset G$ where $\chi(H) > 2^a$, and for all $uv \in E(H)$ we have $\chi(G(u,v)) \geq 2^b$.
\end{lemma}

\begin{proof}
Let $\{r_i\}_{i = 0}^m$ be a $2^b$-sequence of $V(G)$.  Then for $r_0 = 1$ and $r_m = n$, we have subgraphs $G[r_0,r_1], G(r_1,r_2],...,G(r_{m-1},r_m]$, such that

$$\chi(G[r_0,r_1]) = \chi(G(r_1,r_2]) = \cdots = \chi(G(r_{m-2},r_{m-1}]) =  2^b$$

\noindent and

$$\chi(G(r_{m-1},r_m]) \leq 2^b.$$

\noindent For each of these subgraphs, we will properly color its vertices with colors, say, $1,2,...,2^b$.  Since $\chi(G) > 2^{a + b + 1}$, there exists a color class for which the vertices of this color induce a subgraph with chromatic number at least $2^{a  + 1}$.  Let $G'$ be such a subgraph, and we define subgraphs $H_1,H_2\subset G'$ such that

$$H_1 = G'[r_0,r_1]\cup G'(r_2,r_3]\cup \cdots \hspace{1cm}\textnormal{and}\hspace{1cm} H_2 = G'(r_1,r_2]\cup G'(r_3,r_4] \cup \cdots$$

Since $V(H_1)\cup V(H_2) = V(G')$, either $\chi(H_1) > 2^a$ or $\chi(H_2) > 2^a$.  Without loss of generality, we can assume $\chi(H_1) > 2^a$ holds, and set $H = H_1$.  Now for any $uv \in E(H)$, there exists integers $i,j$ such that for $0 \leq i < j$, we have $u \in V(G'(r_{2i},r_{2i + 1}])$ and $v\in V(G'(r_{2j},r_{2j + 1}])$.  This implies $G(r_{2i + 1},r_{2i + 2}] \subset G(u,v)$ and

$$\chi(G(u,v)) \geq \chi(G(r_{2i + 1},r_{2i + 2}] ) = 2^b.$$

\noindent This completes the proof of the lemma.
\end{proof}

\noindent Recall that the distance between two vertices $u,v \in V(G)$ in a graph $G$, is the length of the shortest path from $u$ to $v$.

\begin{lemma}
\label{distance}
Let $G$ be a graph and let $v \in V(G)$.  Suppose $G^0,G^1,G^2,...$ are the subgraphs induced by vertices at distance $0,1,2,...$ respectively from $v$.  Then for some $d$, $\chi(G_d) \geq \chi(G)/2$.
\end{lemma}

\begin{proof}
For $0 \leq i < j$, if $|i - j| > 1$, then no vertex in $G^i$ is adjacent to a vertex in $G^j$.  By the Pigeonhole Principle, the statement follows.

\end{proof}

\section{Using the induction hypothesis}
The proof of Theorem \ref{flag} will be given in Section 5, and is done by induction on $k$.  By setting $\lambda_k = 2^{(5^{k+2} - 121)/4}$, we will show that if $\mathcal{F}$ is a simple family of $x$-monotone right-flag curves with $\chi(\mathcal{F})> 2^{\lambda_k}$, then $\mathcal{F}$ contains $k$ pairwise crossing members.  In the following two subsections, we will assume that the statement holds for fixed $k < k'$, and show that if $\chi(\mathcal{F})> 2^{5\lambda_k + 120}$, $\mathcal{F}$ must contain either $k+1$ pairwise crossing members, or a special subconfiguration.  Note that $\lambda_k$ satisfies $\lambda_k > \log(2k)$ for all $k$.

\subsection{Key lemma}
Let $\mathcal{F} = \{C_1,C_2,...,C_n\}$ be a simple family of $n$ $x$-monotone right-flag curves, such that no $k+1$ members pairwise cross.  Suppose that curves $C_a$ and $C_b$ intersect, for $a < b$.  Let

\begin{eqnarray*}
\mathcal{I}_a & =& \{ C_i \in \mathcal{F}(a,b): \textnormal{$C_i$ intersects $C_a$}\},\\\\
\mathcal{D}_a & =& \{C_i \in \mathcal{F}(a,b): \textnormal{$C_i$ does not intersect $C_a$}\}.
\end{eqnarray*}

\noindent We define $\mathcal{I}_b$ and $\mathcal{D}_b$ similarly and set $\mathcal{D}_{ab} = \mathcal{D}_a\cap \mathcal{D}_b \subset \mathcal{F}(a,b)$.  Now we define three subsets of $\mathcal{D}_{ab}$ as follows:

\begin{eqnarray*}
\mathcal{D}_{ab}^a & =& \{C_i \in \mathcal{D}_{ab}: \textnormal{$\exists C_j \in \mathcal{I}_a$ that intersects $C_i$}\},\\\\
\mathcal{D}_{ab}^b & =& \{C_i \in \mathcal{D}_{ab}: \textnormal{$\exists C_j \in \mathcal{I}_b$ that intersects $C_i$}\},\\\\
\mathcal{D} &= &\mathcal{D}_{ab}\setminus(\mathcal{D}_{ab}^a\cup \mathcal{D}_{ab}^b).
\end{eqnarray*}

\noindent We now prove the following key lemma.

\begin{lemma}
\label{crucial}
 $\chi(\mathcal{D}_{ab}^a\cup \mathcal{D}_{ab}^b) \leq k\cdot 2^{2\lambda_k + 102}$.  Hence, $\chi(\mathcal{D}) \geq \chi(\mathcal{F}(a,b)) - 2^{\lambda_k+1} -  k\cdot 2^{2\lambda_k + 102}$.

\end{lemma}

\begin{proof}
Without loss of generality, we can assume that

\begin{equation}
\chi(\mathcal{D}_{ab}^a) \geq \frac{\chi(\mathcal{D}_{ab}^a\cup \mathcal{D}_{ab}^b) }{2},
\end{equation}

\noindent since otherwise a similar argument will follow if $\chi(\mathcal{D}_{ab}^b) \geq \chi(\mathcal{D}_{ab}^a\cup \mathcal{D}_{ab}^b)/2$.  Let $\mathcal{D}_{ab}^a = \{C_{r_1},C_{r_2},...,C_{r_{m_1}}\}$ where $a < r_1 < r_2 < \cdots < r_{m_1} < b$.  For each curve $C_i \in \mathcal{I}_a$, we define $A_{i}$ to be the arc along the curve $C_{i}$, from the left endpoint of $C_{i}$ to the intersection point $C_{i}\cap C_a$.  Set $\mathcal{A} = \{A_i: \mathcal{C}_i \in \mathcal{I}_a\}$.  Notice that the intersection graph of $\mathcal{A}$ is an incomparability graph, which is a perfect graph (see \cite{fox2}).  Since there are no $k+1$ pairwise crossing arcs in $\mathcal{A}$, we can decompose the members in $\mathcal{A}$ into $k$ parts $\mathcal{A} = \mathcal{A}_1\cup \mathcal{A}_1\cup \cdots\cup \mathcal{A}_k$, such that the arcs in $\mathcal{A}_i$ are pairwise disjoint.  Then for $i = 1,2,...,k$, we define

$$\mathcal{S}_i = \{C_j \in \mathcal{D}_{ab}^a: \textnormal{$C_{j}$ intersects an arc from $\mathcal{A}_{i}$}\}.$$

\noindent Since $\mathcal{D}_{ab}^a = \mathcal{S}_1 \cup \mathcal{S}_2 \cup \cdots \cup \mathcal{S}_k$, there exists a $t \in \{1,2,...,k\}$ such that

\begin{equation}
\chi(\mathcal{S}_t) \geq \frac{\chi(\mathcal{D}_{ab}^a)}{k}.
\end{equation}

\noindent Therefore, let $\mathcal{A}_{t}  =  \{A_{p_1},A_{p_2},...,A_{p_{m_2}}\}$.  Notice that each curve $C_{i} \in S_t$ intersects the members in $\mathcal{A}_t$ that lies either above or below $C_{i}$ (but not both since $\mathcal{F}$ is simple).  Moreover, $C_{i}$ intersects the members in $\mathcal{A}_t$ in either increasing or decreasing order.  Let $\mathcal{S}_t^{1}$ ($\mathcal{S}_t^2$) be the curves in $\mathcal{S}_t$ that intersects a member in $\mathcal{A}_t$ that lies above (below) it.  Again, without loss of generality we will assume that

\begin{equation}
\chi(\mathcal{S}_t^1) \geq   \frac{\chi(\mathcal{S}_t)}{2},
\end{equation}

\noindent since a symmetric argument will hold if $\chi(\mathcal{S}_t^2) \geq   \chi(\mathcal{S}_t)/2$.  For each curve $C_{i} \in \mathcal{S}^1_{t}$, we define

\begin{eqnarray*}
u(i) & = & \max\{j: \textnormal{arc $A_{p_j}\in \mathcal{A}_t$ intersects $C_{i}$}\},\\\\
l(i) & = & \min\{j: \textnormal{arc $A_{p_j} \in \mathcal{A}_t$ intersects $C_{i}$}\}.
\end{eqnarray*}

\noindent See Figure \ref{example} for a small example.  Then for each curve $C_{i} \in  \mathcal{S}^1_{t}$, we define the curve $L_{i}$ to be the arc along $C_{i}$, joining the left endpoint of $C_{i}$ and the point $C_{i}\cap A_{p_{u(i)}}$.  See Figure \ref{li}.  Now we set

 \begin{figure}
  \centering
\subfigure[In this example, $l(i) = 2$ and $u(i) = 3$.]{\label{example}\includegraphics[width=0.31\textwidth]{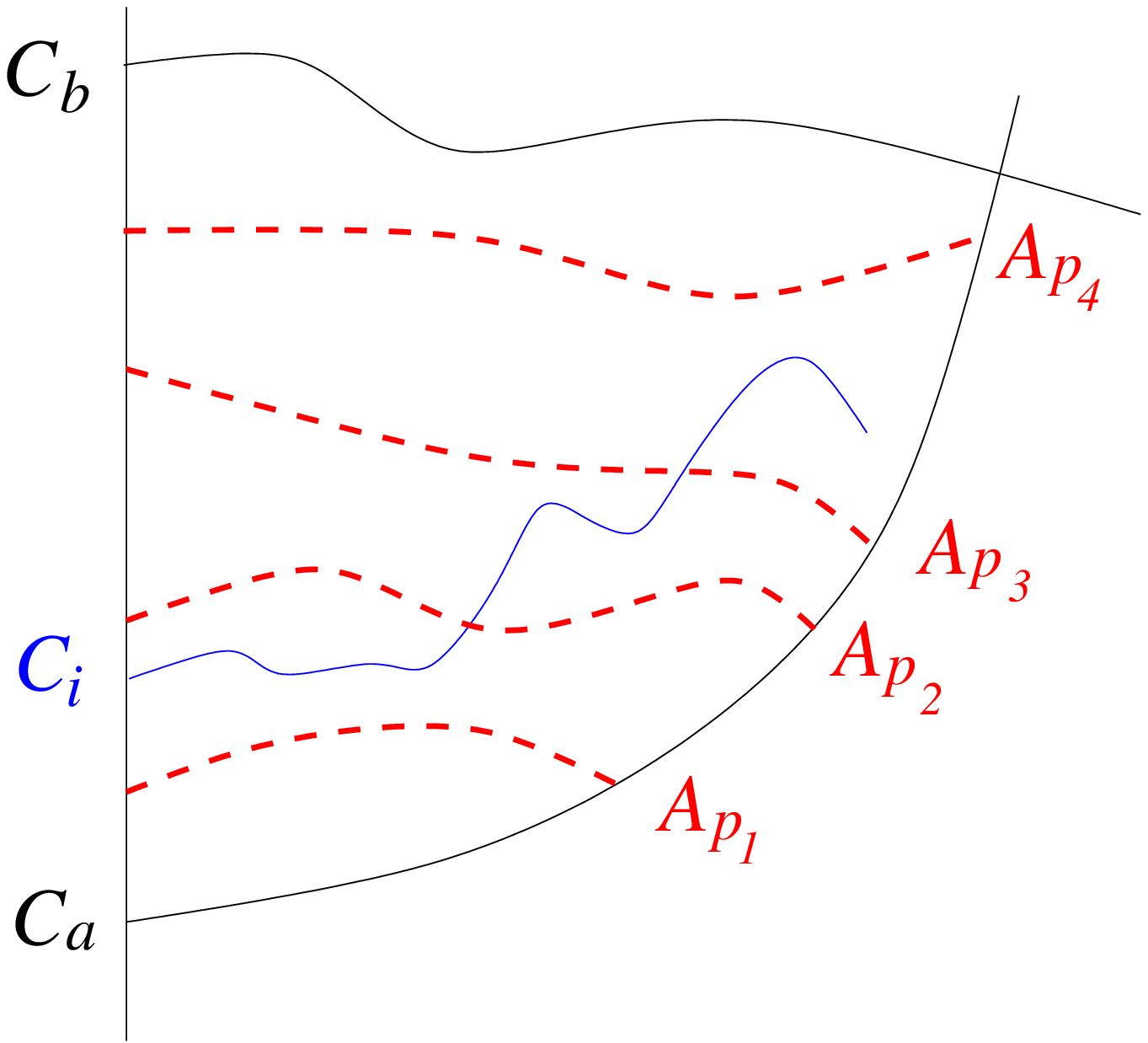}}\hspace{2cm}
\subfigure[$L_i$ drawn thick.]{\label{li}\includegraphics[width=0.31\textwidth]{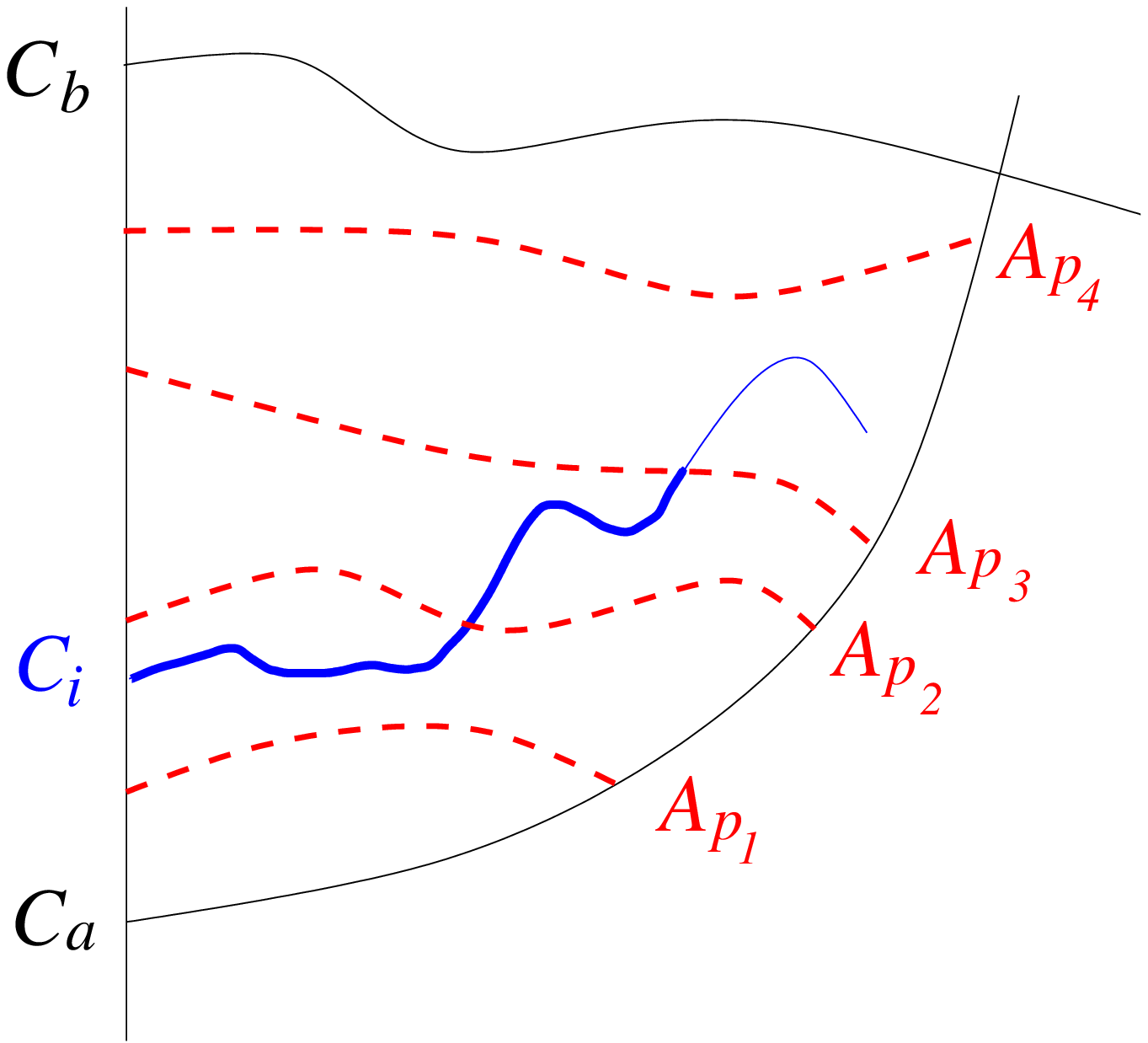}}
                        \caption{Curves in $\mathcal{S}_t^1$.}
  \label{lu}
\end{figure}

 $$\mathcal{L} = \{L_i: C_i \in \mathcal{S}_t^1\}.$$

\noindent  Notice that $\mathcal{L}$ does not contain $k$ pairwise crossing members.  Indeed, otherwise these $k$ curves $\mathcal{K}\subset \mathcal{L}$ would all intersect $A_{p_{i}}$ where

$$i =  \min\limits_{C_j \in \mathcal{K}} u(j),$$

\noindent  creating $k+1$ pairwise crossing curves in $\mathcal{F}$.  Therefore we can decompose $\mathcal{L} = \mathcal{L}_1\cup \mathcal{L}_2\cup \cdots \cup \mathcal{L}_{w}$ into $w$ parts, such that $w \leq 2^{\lambda_k}$, and the set of curves in $\mathcal{L}_i$ are pairwise disjoint for $i = 1,2,...,w$.  Let $\mathcal{H}_i\subset \mathcal{F}$ be the set of (original) curves corresponding to the (modified) curves in $\mathcal{L}_i$.  Then there exists an $s \in \{1,2,...,w\}$ such that

\begin{equation}
\chi(\mathcal{H}_s) \geq \frac{\chi(\mathcal{S}_t^1) }{2^{\lambda_k}}.
\end{equation}

\noindent Now for each curve $C_i \in \mathcal{H}_s$, we will define the curve $U_i$ as follows.  Let $T_i$ be the arc along $C_i$, joining the right endpoint of $C_i$ and the point $C_i\cap A_{p_{l(i)}}$.  We define $B_i$ to be the arc along $C_i$, joining the left endpoint of $C_i$ and the point $C_i\cap A_{p_{l(i)}}$.  See Figure \ref{tbi}.  Notice that for any two curves $C_i,C_j \in \mathcal{H}_s$, $B_i$ and $B_j$ are disjoint.  We define $U_i = B'_i\cup T_i$, where $B'_i$ is the arc obtained by pushing $B_i$ upwards toward $A_{p_{l(i)}}$,  such that $B'_i$ does not introduce any new crossing points, and $B'_i$ is ``very close" to the curve $A_{p_{l(i)}}$, meaning that no curve in $\mathcal{H}_s$ has its right endpoint in the region enclosed by $B'_i, A_{p_{l(i)}}$, and the $y$-axis.  See Figure \ref{bprime}.  We do this for every curve $C_i \in \mathcal{H}_s$, to obtain the family $\mathcal{U} = \{U_i: C_i \in \mathcal{H}_s\}$, such that

\begin{enumerate}
\item $\mathcal{U}$ is a simple family of $x$-monotone right-flag curves,

 \item we do not create any new crossing pairs in $\mathcal{U}$ (we may lose some crossing pairs),

\item any curve $U_j$ that crosses $B'_i$, must cross $A_{p_{l(i)}}$.
\end{enumerate}

 \begin{figure}[h]
  \centering
\subfigure[$T_i$ drawn thick and $B_i$ drawn normal along $C_i$.]{\label{tbi}\includegraphics[width=0.31\textwidth]{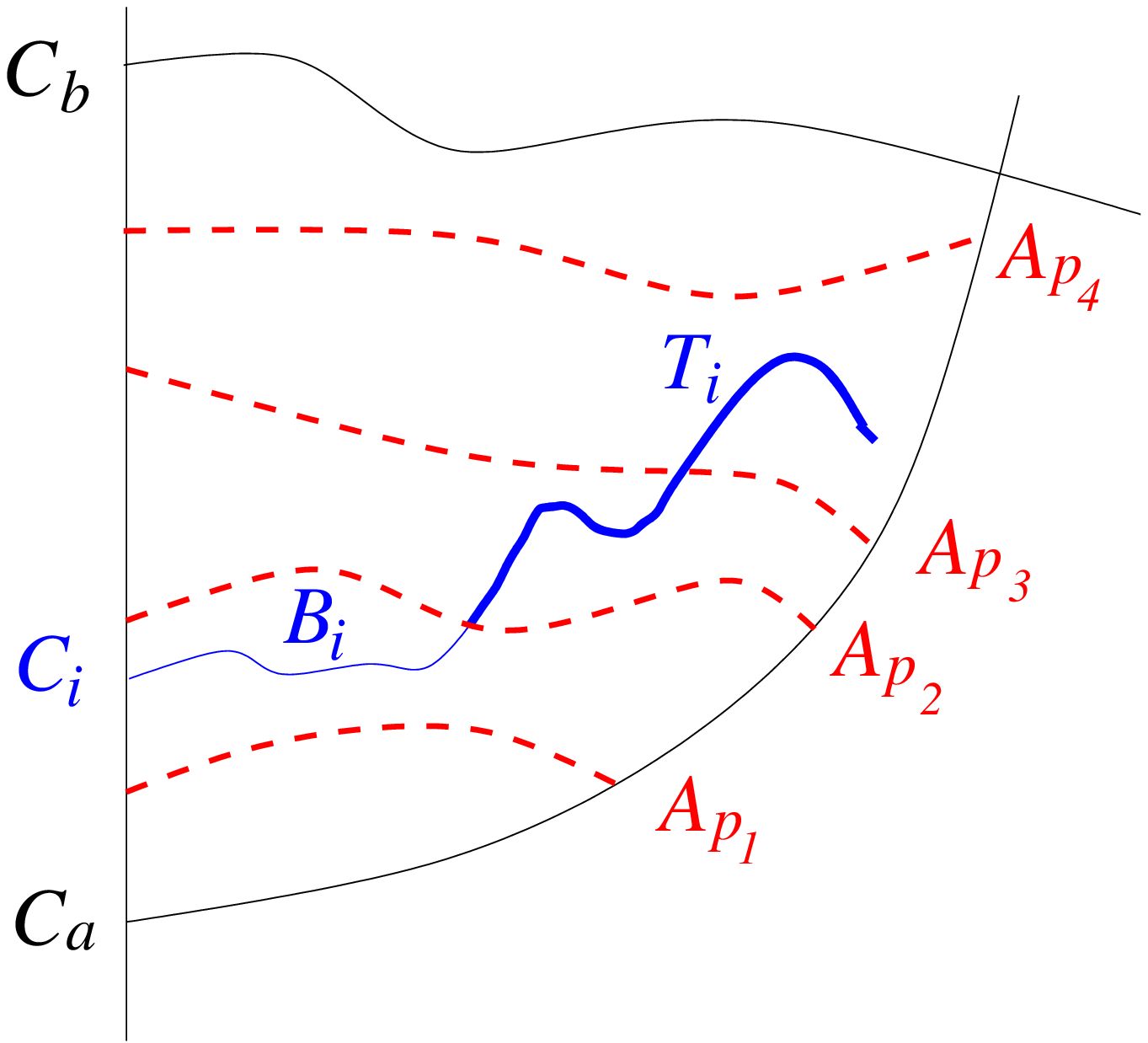}}\hspace{2cm}
\subfigure[Curve $U_i = B'_i \cup T_i$.]{\label{bprime}\includegraphics[width=0.31\textwidth]{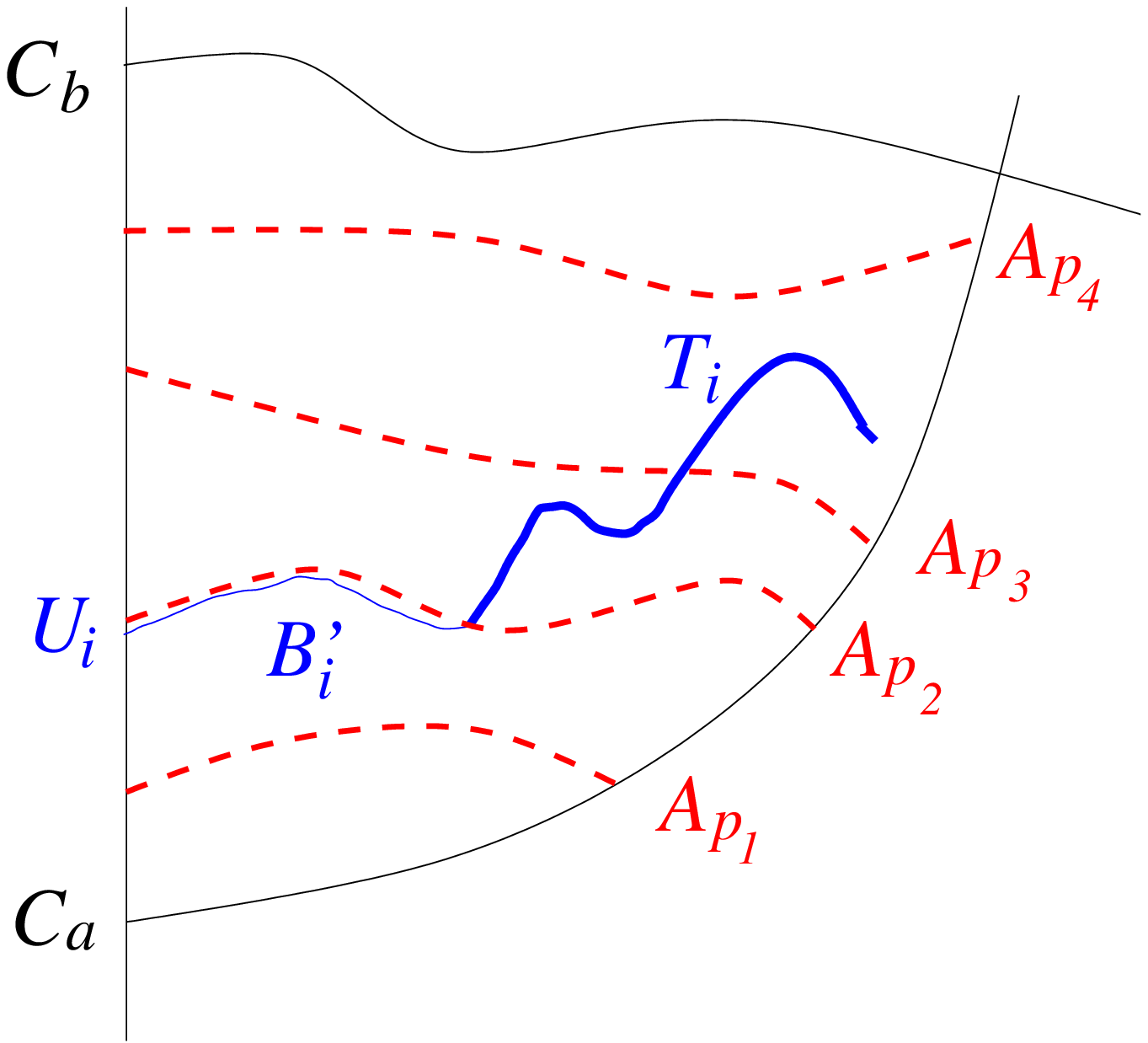}}
                        \caption{Curves in $\mathcal{S}_t^1$.}
  \label{lu}
\end{figure}

Notice that $\mathcal{U}$ does not contain $k$ pairwise crossing members.  Indeed, otherwise these $k$ curves $\mathcal{K} \subset \mathcal{U}$ would all cross $A_{p_i}$ where

$$i = \max\limits_{U_j \in \mathcal{K}}l(j),$$

\noindent creating $k+1$ pairwise crossing curves in $\mathcal{F}$.  Therefore we can decompose $\mathcal{U} = \mathcal{U}_1\cup \mathcal{U}_2\cup \cdots \cup \mathcal{U}_z$ into $z$ parts, such that $z \leq 2^{\lambda_k}$ and the curves in $\mathcal{U}_i$ are pairwise disjoint.  Let $\mathcal{C}_i\subset \mathcal{F}$ be the set of (original) curves corresponding to the (modified) curves in $\mathcal{U}_i$.  Then there exists an $h \in \{1,2,...,z\}$ such that

\begin{equation}
\chi(\mathcal{C}_h) \geq \frac{\chi(\mathcal{H}_s)}{2^{\lambda_k}}.
\end{equation}

\noindent Now we make the following observation.

\begin{observation}

There are no three pairwise crossing curves in $\mathcal{C}_h$.

\end{observation}
\begin{proof}
Suppose that the pair of curves $C_i,C_j \in \mathcal{C}_h$ intersect, for $i < j$.  Then we must have $i < p_{u(i)} < j < p_{l(j)}$ and $A_{p_{u(i) + 1}} = A_{p_{l(j)}}$.  Basically the ``top tip" of $C_i$ must intersect the ``bottom tip" of $C_j$.  See Figure \ref{example2}.  Hence if $C_i$ crosses $C_j$ and $C_k$ for $i < j < k$, then $p_{u(i)} < j < k < p_{l(j)}$, and therefore $C_j$ and $C_k$ must be disjoint.  This implies that $\mathcal{C}_h$ does not contain three pairwise crossing members.

\begin{figure}[h]
\begin{center}
\includegraphics[width=160pt]{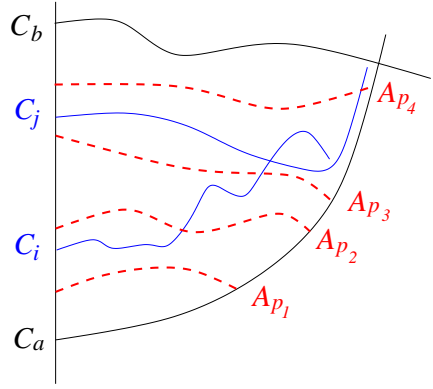}
  \caption{$C_i$ and $C_j$ cross.}
  \label{example2}
 \end{center}
\end{figure}

\end{proof}

\noindent By a result of McGuinness \cite{sean}, we know that

\begin{equation}
\chi(\mathcal{C}_h) \leq 2^{100}.
\end{equation}

\noindent  Therefore, by combining equations (1),(2),(3),(4),(5), and (6), we have

$$\chi(\mathcal{D}_{ab}^a\cup \mathcal{D}_{ab}^b) \leq  k\cdot 2^{2\lambda_k + 102}.$$

\noindent Since neither $\mathcal{I}_a$ nor $\mathcal{I}_b$ contain $k$ pairwise crossing members, we have $\chi(\mathcal{I}_a),\chi(\mathcal{I}_b) \leq 2^{\lambda_k}$.  Therefore

\begin{eqnarray*}
\chi(\mathcal{D}) & \geq &\chi(\mathcal{F}(a,b)) - \chi(\mathcal{I}_a) - \chi(\mathcal{I}_b) - \chi(\mathcal{D}^a_{ab}\cup \mathcal{D}^b_{ab})\\\\
 &\geq &\chi(\mathcal{F}(a,b)) -  2^{\lambda_k +1} -  k\cdot 2^{2\lambda_k + 102}.
\end{eqnarray*}

\end{proof}

Therefore, if there exists a curve $C_i \in \mathcal{F}$ that intersects $C_a$ (or $C_b)$ and a curve from $\mathcal{D}$, then $i < a$ or $i > b$.

\subsection{Finding special configurations}
In this section, we will show that if the chromatic number of $G(\mathcal{F})$ is sufficiently high, then certain subconfigurations must exist.  We say that the set of curves $\{C_{i_1},C_{i_2},...,C_{i_{k + 1}}\}$ forms a \emph{type 1 configuration}, if

\begin{enumerate}

\item $i_1 < i_2 < \cdots < i_{k + 1}$,

 \item the set of $k$ curves $\mathcal{K} = \{C_{i_1},C_{i_2},...,C_{i_k}\}$ pairwise intersects,

 \item $C_{i_{k + 1}}$ does not intersect any of the curves in $\mathcal{K}$, and

 \item $x(C_{i_{k + 1}}) < x(\mathcal{K})$.  See Figure \ref{type1}.
\end{enumerate}

\noindent Likewise, we say that the set of curves $\{C_{i_1},C_{i_2},...,C_{i_{k + 1}}\}$ forms a \emph{type 2 configuration}, if

\begin{enumerate}

\item $i_1 < i_2 < \cdots < i_{k + 1}$,

 \item the set of $k$ curves $\mathcal{K} = \{C_{i_2},C_{i_3},...,C_{i_{k + 1}}\}$ pairwise intersects,

 \item $C_{i_{ 1}}$ does not intersect any of the curves in $\mathcal{K}$, and

 \item $x(C_{i_{1}}) < x(\mathcal{K})$.  See Figure \ref{type2}.
\end{enumerate}

\noindent We say that the set of curves $\{C_{i_1},C_{i_2},...,C_{i_{2k + 1}}\}$ forms a \emph{type 3 configuration}, if

\begin{enumerate}
\item $i_1 < i_2 < \cdots < i_{2k+1}$,

 \item the set of $k$ curves $\mathcal{K}_1 = \{C_{i_1},...,C_{i_k}\}$ pairwise intersects,

 \item the set of $k$ curves $\mathcal{K}_2 = \{C_{i_{k+2}},C_{i_{k + 3}},...,C_{i_{2k + 1}}\}$ pairwise intersects,

\item $C_{i_{k + 1}}$ does not intersect any of the curves in $\mathcal{K}_1\cup \mathcal{K}_2$, and

\item $x(C_{i_{k + 1}}) \leq x(\mathcal{K}_1\cup \mathcal{K}_2)$.  See Figure \ref{type3}.

\end{enumerate}

 \begin{figure}
  \centering
\subfigure[Type 1 configuration.]{\label{type1}\includegraphics[width=0.3\textwidth]{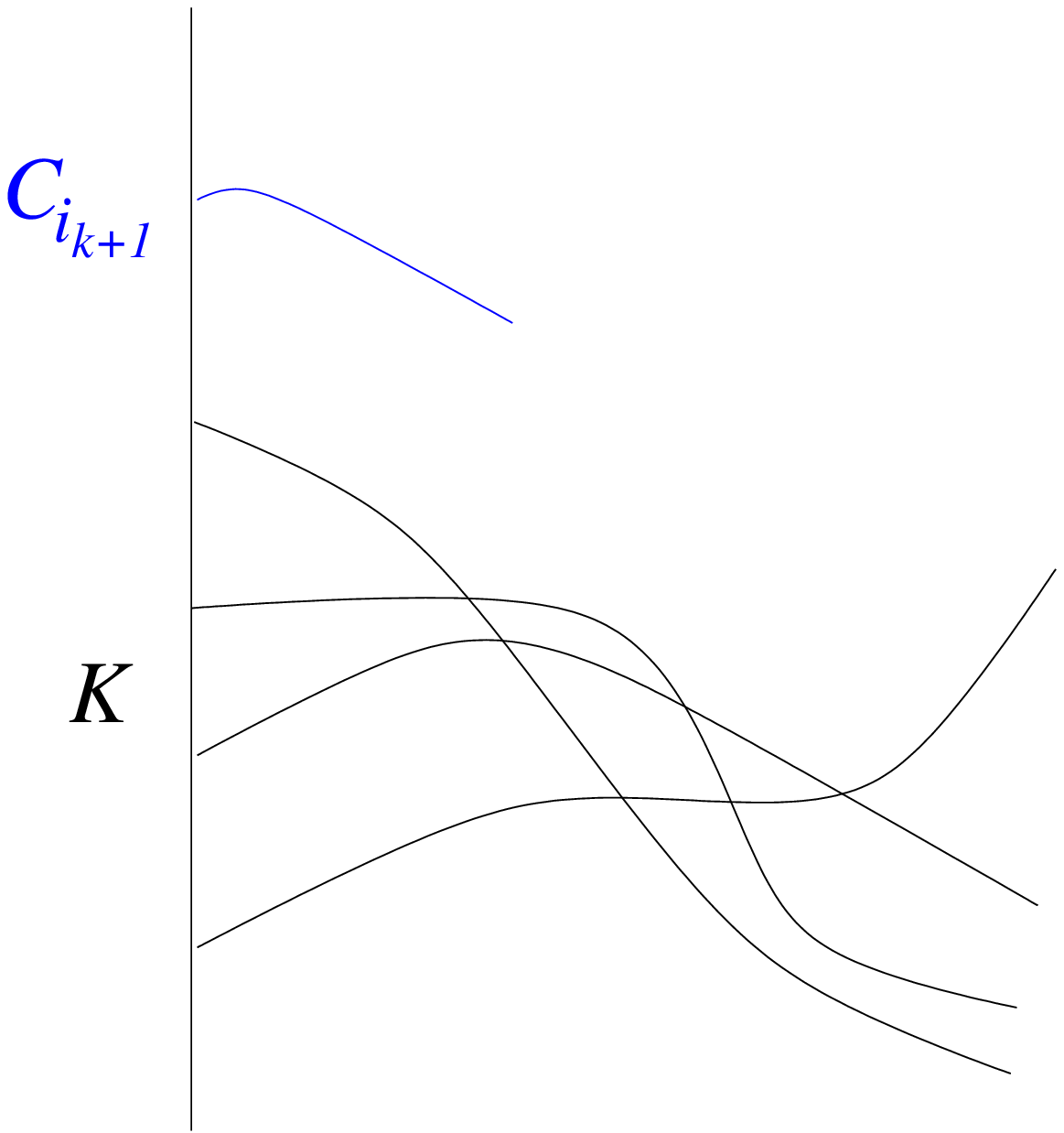}}\hspace{.7cm}
\subfigure[Type 2 configuration.]{\label{type2}\includegraphics[width=0.3\textwidth]{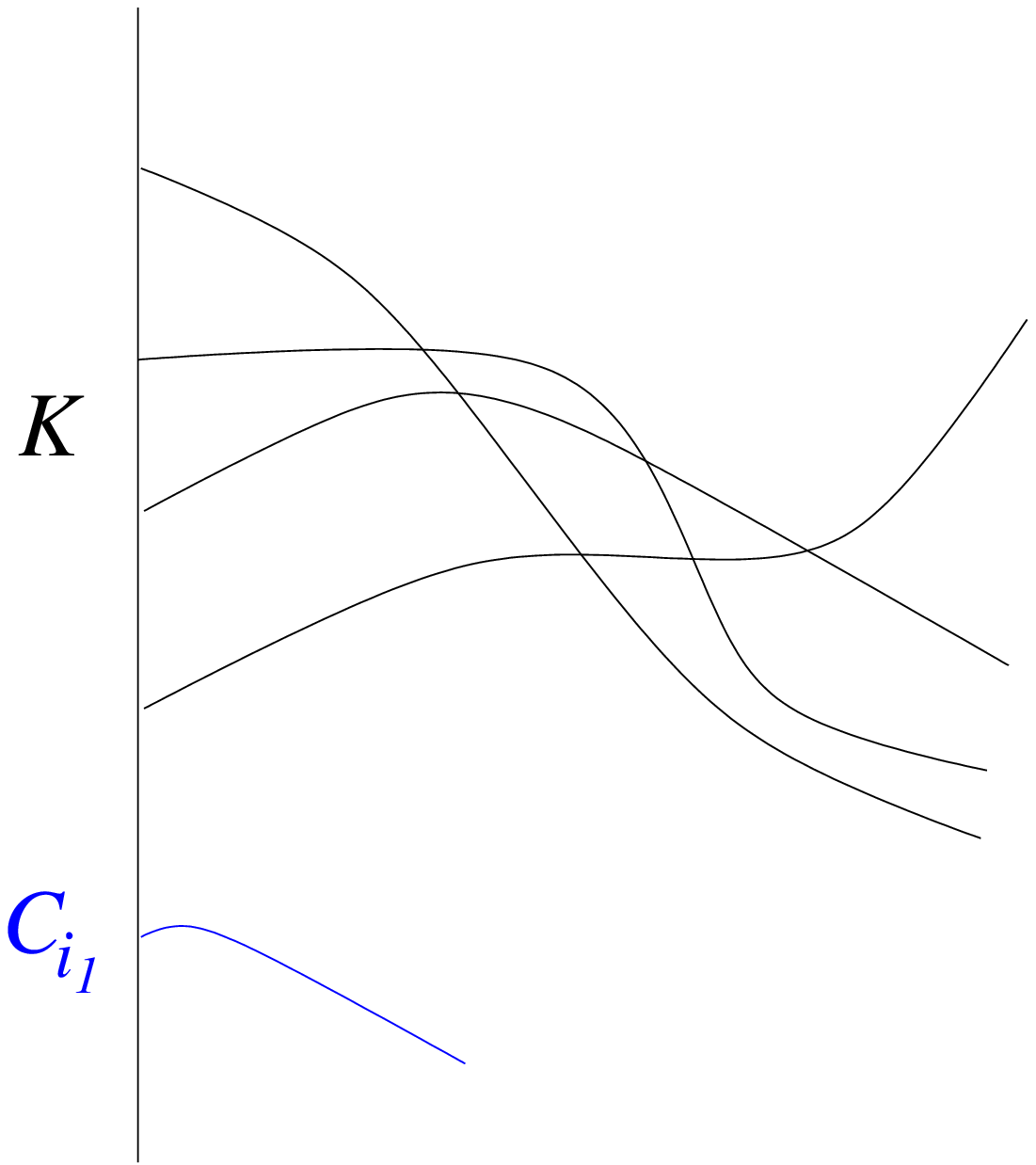}}\hspace{.7cm}
\subfigure[Type 3 configuration.]{\label{type3}\includegraphics[width=0.3\textwidth]{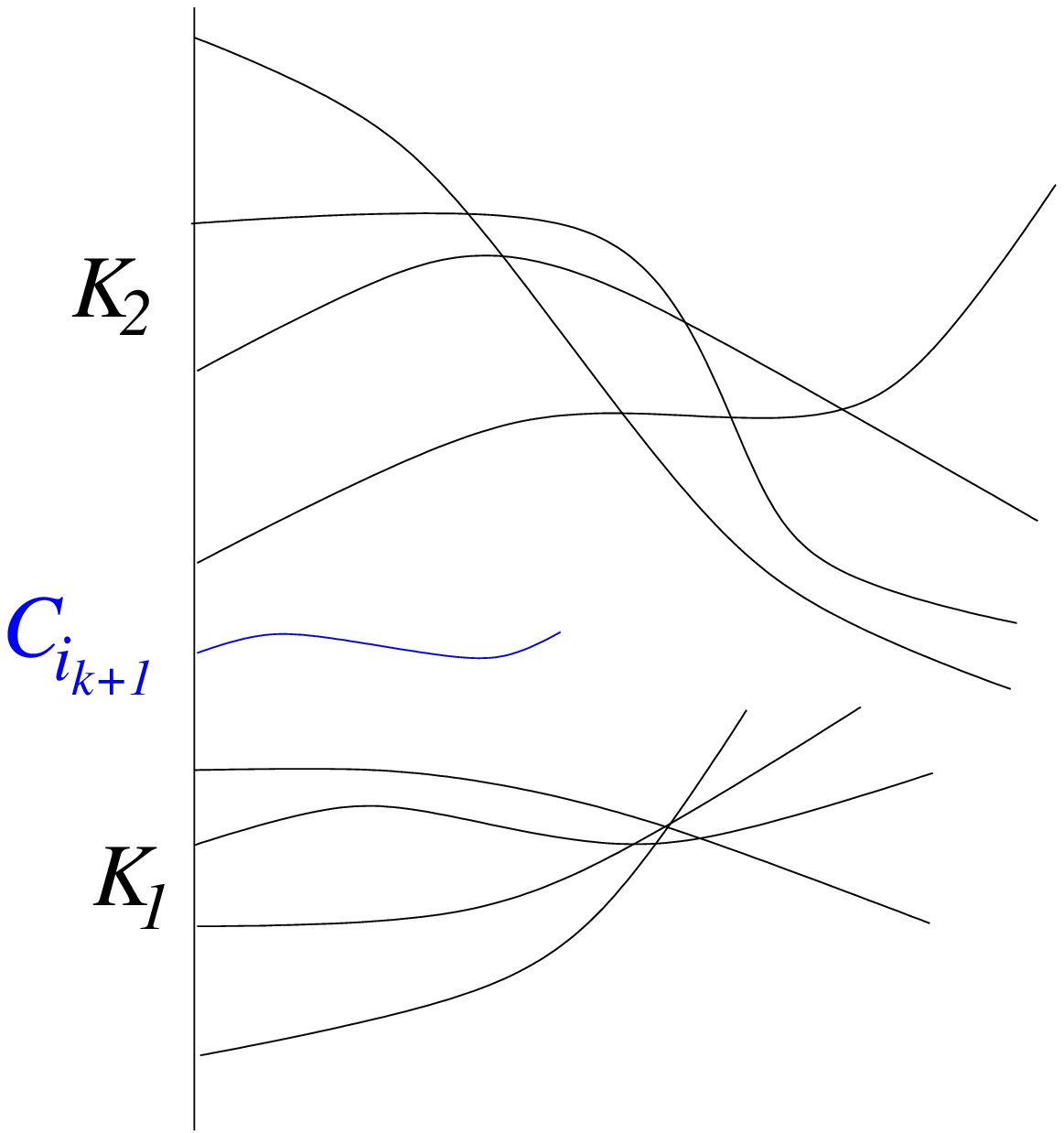}}
                        \caption{Special configurations.}
  \label{lu}
\end{figure}

\noindent Note that in a type 3 configuration, a curve in $\mathcal{K}_1$ may or may not intersect a curve in $\mathcal{K}_2$.  The goal of this subsection will be to show that if $G(\mathcal{F})$ has large chromatic number, then it must contain a type 3 configuration.  We start by proving several lemmas.

\begin{lemma}
\label{short}
Let $\mathcal{F} = \{C_1,C_2,...,C_n\}$ be a family of $n$ $x$-monotone right-flag curves.  Suppose the set of curves $\mathcal{K} = \{C_{i_1},C_{i_2},...,C_{i_m}\}$ pairwise intersect with $i_1 < i_2 < \cdots < i_m$.  If there exits a curve $C_j$ such that $C_j$ is disjoint to all members in $\mathcal{K}$ and $i_1 < j < i_m$, then

$$x(C_j) \leq x(\mathcal{K}).$$

\end{lemma}

\begin{proof}

Suppose that $x(C_{i_t}) < x(C_j)$ for some $t$.  Without loss of generality, we can assume $i_1< j < i_t$.  Since $C_{i_1}$ and $C_{i_t}$ cross and are $x$-monotone, this implies that either $C_{i_1}$ or $C_{i_t}$ intersects $C_j$ and therefore we have a contradiction.  A symmetric argument holds if $i_t < j < i_m$.

\end{proof}

\begin{lemma}
\label{one}
Let $\mathcal{F} = \{C_1,C_2,...,C_n\}$ be a family of $n$ $x$-monotone right-flag curves.  Then for any set of $t$ curves $C_{i_1},C_{i_2},...,C_{i_t}\in \mathcal{F}$ where $ t\leq 2k$, if $\chi(\mathcal{F}) > 2^{\beta}$, then either

\begin{enumerate}

\item $\mathcal{F}$ contains $k+ 1$ pairwise crossing members, or

\item there exists a subset $\mathcal{H}\subset \mathcal{F}\setminus\{C_{i_1},C_{i_2},...,C_{i_t}\}$, such that each curve $C_j \in \mathcal{H}$ is disjoint to all members in $\{C_{i_1},C_{i_2}...,C_{i_t}\}$, and $\chi(\mathcal{H}) > 2^{\beta} - 2^{2\lambda_k}$.
\end{enumerate}
\end{lemma}

\begin{proof}

For each $j \in \{1,2,...,t\}$, let $\mathcal{H}_{j} \subset \mathcal{F}$ be the subset of curves that intersect $C_{i_j}$.  If $\chi(\mathcal{H}_{j}) > 2^{\lambda_k}$ for some $j\in \{1,2,...,t\}$, then $\mathcal{F}$ contains $k+1$ pairwise crossing members.  Therefore, we can assume that $\chi(\mathcal{H}_{j}) \leq 2^{\lambda_k}$ for all $1\leq j \leq t$.  Now let $\mathcal{H}\subset \mathcal{F}$ be the subset of curves defined by

$$\mathcal{H} = \mathcal{F}\setminus(\mathcal{H}_{1}\cup\mathcal{H}_{2}\cup \cdots \cup \mathcal{H}_{t}).$$

\noindent Since $\chi(\mathcal{F}) > 2^{\beta}$, we have

$$\chi(\mathcal{H}) > 2^{\beta} - t2^{\lambda_k} \geq 2^{\beta} - 2^{2\lambda_k},$$

\noindent where the last inequality follows from the fact that $\log t < \log 2k < \lambda_k$.

\end{proof}

\begin{lemma}

\label{two}
Let $\mathcal{F} = \{C_1,C_2,...,C_n\}$ be a family of $n$ $x$-monotone right-flag curves.  If $\chi(\mathcal{F}) \geq 2^{4\lambda_k + 107}$, then either

\begin{enumerate}

\item $\mathcal{F}$ contains $k+1$ pairwise crossing members, or

\item $\mathcal{F}$ contains a type 1 configuration, or

    \item $\mathcal{F}$ contains a type 2 configuration.
\end{enumerate}
\end{lemma}

\begin{proof}
Assume that $\mathcal{F}$ does not contain $k+1$ pairwise crossing members.  By Lemma \ref{distance}, for some $d \geq 2$, the subset of curves $\mathcal{F}^d$ at distance $d$ from curve $C_1$ satisfies

$$\chi(\mathcal{F}^d) \geq \frac{\chi(\mathcal{F})}{2} \geq 2^{4\lambda_k + 106}.$$

By Lemma \ref{sequence}, there exists a subset $\mathcal{H}_1 \subset \mathcal{F}^d$ such that $\chi(\mathcal{H}_1) >2$, and for every pair of curves $C_a,C_b\in \mathcal{H}_1$ that intersect, $\mathcal{F}^d(a,b)\geq 2^{4\lambda_k + 104}$.  Fix two such curves $C_a,C_b \in \mathcal{H}_1$ and let $\mathcal{A}$ be the set of curves in $\mathcal{F}(a,b)$ that intersects either $C_a$ or $C_b$.  By Lemma \ref{crucial}, there exists a subset $\mathcal{D}_1\subset \mathcal{F}^d(a,b)$ such that each curve $C_i \in \mathcal{D}_1$ is disjoint to $C_a$, $C_b$, $\mathcal{A}$, and moreover

 $$\chi(\mathcal{D}_1) \geq 2^{4\lambda_k + 104} -2^{\lambda_k + 1} -  k\cdot 2^{2\lambda_k + 102} > 2^{4\lambda_k + 103}.$$

\noindent Again by Lemma \ref{sequence}, there exists a subset $\mathcal{H}_2 \subset \mathcal{D}_1$ such that $\chi(\mathcal{H}_2) > 2^{\lambda_k}$, and for each pair of curves $C_u,C_v \in \mathcal{H}$ that intersect, $\chi(\mathcal{D}_1(u,v)) \geq 2^{3\lambda_k + 102}$.  Therefore, $\mathcal{H}_2$ contains $k$ pairwise crossing curves $C_{i_1},...,C_{i_k}$ such that $i_1 < i_2 < \cdots < i_k$.  Since $\chi(\mathcal{D}_1(i_1,i_2)) \geq 2^{3\lambda_k + 102}$, by Lemma \ref{one}, there exists a subset $\mathcal{D}_2 \subset \mathcal{D}(i_1,i_2)$ such that every curve $C_l \in \mathcal{D}_2$ is disjoint to the set of curves $\{C_{i_1},...,C_{i_k}\}$ and

$$\chi(\mathcal{D}_2) \geq 2^{3\lambda_k + 102} - 2^{2\lambda_k} > 2^{3\lambda_k + 101}.$$

\noindent By applying Lemma \ref{sequence} one last time, there exists a subset $\mathcal{H}_3 \subset \mathcal{D}_2$ such that $\chi(\mathcal{H}_3) > 2^{\lambda_k}$, and for every pair of curves $C_u,C_v \in \mathcal{H}_3$ that intersect, we have $\chi(\mathcal{D}_2(u,v)) \geq 2^{2\lambda_k + 100}$.  Therefore, $\mathcal{H}_3$ contains $k$ pairwise intersecting curves $C_{j_1},C_{j_2},...,C_{j_k}$ such that $i_1 < j_1 < j_2 < \cdots < j_k < i_2$.  Since $\chi(\mathcal{D}_2(j_{k-1},j_k)) \geq 2^{2\lambda_k + 100}$, by Lemma \ref{one}, there exists a subset $\mathcal{D}_3 \subset \mathcal{D}_2(j_{k-1},j_k)$ such that every curve $C_l \in \mathcal{D}_3$ is disjoint to the set of curves $\{C_{j_1},...,C_{j_k}\}$ (and disjoint to the set of curves $\{C_{i_1},C_{i_2},...,C_{i_k}\}$) and

$$\chi(\mathcal{D}_3) \geq 2^{2\lambda_k  + 100} - 2^{2\lambda_k}>  2^{2\lambda_k + 99}.$$

\noindent Now we can define a $2^{2\lambda_k + 1}$-sequence $\{r_i\}_{i = 0}^m$ of $\mathcal{D}_3$ such that $m \geq 4$ (recall the definition of an $\alpha$-sequence of $\mathcal{F}$ from Section 2).  That is, we have subsets $\mathcal{D}_3[r_0,r_1], \mathcal{D}_3(r_1,r_2], ... , \mathcal{D}_3(r_{m-1},r_m]$ that satisfies

\begin{enumerate}

\item $j_{k-1} < r_0 < r_1 < \cdots < r_{m} < j_k$, and

\item $\chi(\mathcal{D}_3[r_0,r_1]) = \chi(\mathcal{D}_3(r_1,r_2]) = \cdots = \chi(\mathcal{D}_3(r_{m-2},r_{m-1}]) = 2^{2\lambda_k  + 1}$.

\end{enumerate}

\noindent Fix a curve $C_q \in \mathcal{D}_3(r_1,r_2]$.  See Figure \ref{fig}.

 \begin{figure}
  \centering
\subfigure[Curve $C_q$.]{\label{fig}\includegraphics[width=0.39\textwidth]{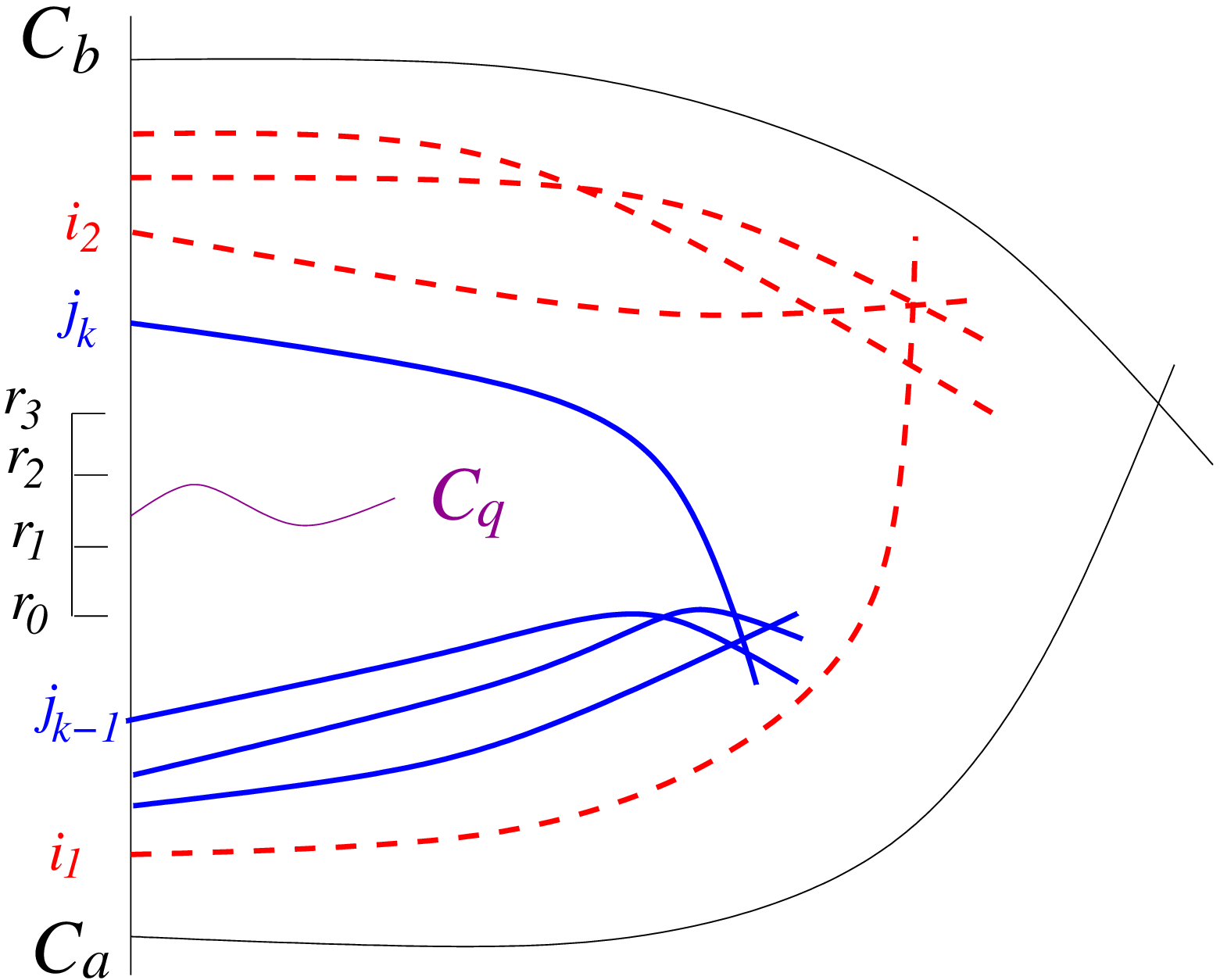}}\hspace{1cm}
\subfigure[Curve $C_{q'}$.]{\label{fig2}\includegraphics[width=0.41\textwidth]{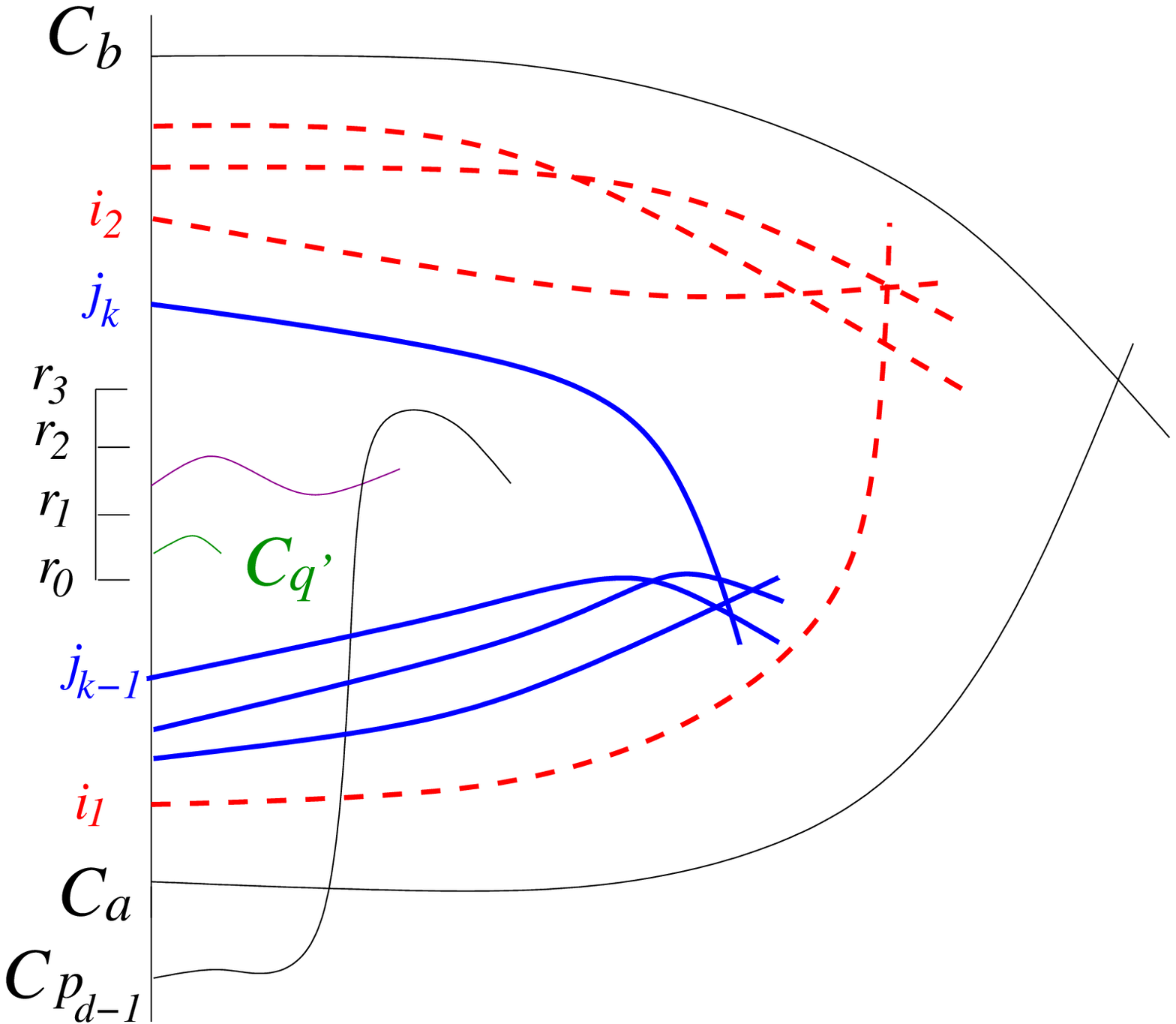}}
                        \caption{Lemma \ref{two}.}
\end{figure}

\medskip

\noindent Since $C_q \in  \mathcal{D}_2(r_1,r_2]\subset \mathcal{F}^d$, there is a path $C_1,C_{p_1},C_{p_2},...,C_{p_{d-1}},C_{q}$ such that $C_{p_t}$ is at distance $t$ from $C_1$ for $1 \leq t \leq d-1$.  Let $R$ be the region enclosed by the $y$-axis, $C_{a}$ and $C_{b}$.  Since $C_q$ lies inside of $R$, and $C_1$ lies outside of $R$, there must be a curve $C_{p_t}$ that intersects either $C_{a}$ or $C_{b}$ for some $1 \leq t \leq d-1$.  Since $C_a,C_b \in \mathcal{F}^d$ and $C_q \in \mathcal{D}_1$, $C_{p_{d-1}}$ must be this curve and we must have either $p_{d-1} < a$ or $p_{d-1} > b$.  Now the proof splits into two cases.

\medskip

\noindent \emph{Case 1.}  Suppose $p_{d-1} < a$.  By Lemma \ref{short}, we have

  $$x(C_q) \leq x(\{C_{j_1},C_{j_2},...,C_{j_{k-1}}\}),$$

 \noindent which implies that the set of $k$ curves $\mathcal{K}  = \{C_{p_{d-1}},C_{j_1},C_{j_2},...,C_{j_{k-1}}\}$ are pairwise crossing.  Now recall that $\chi(\mathcal{D}_3[r_0,r_1]) = 2^{2\lambda_k + 1}$.  By Lemma \ref{one}, there exists a curve $C_{q'}\in \mathcal{D}_3[r_0,r_1]$ such that the curve $C_{q'}$ intersects neither $C_{p_{d-1}}$ nor $C_{q}$.  See Figure \ref{fig2}.  By construction of $\mathcal{D}_3[r_0,r_1]$, $C_{q'}$ does not intersect any of the curves in the set $\mathcal{K}\cup \{C_q\}$.  By Lemma \ref{short}, we have

$$x(C_{q'}) \leq x(C_{q}) \leq x(\mathcal{K}).$$

\noindent Hence $\mathcal{K}\cup \{C_{q'}\}$ is a type 1 configuration.

\medskip

\noindent\emph{Case 2.}  If $p_{d-1} > b$, then by a symmetric argument, $\mathcal{F}$ contains a type 2 configuration.

\end{proof}

\begin{lemma}

\label{three}
Let $\mathcal{F} = \{C_1,C_2,...,C_n\}$ be a family of $n$ $x$-monotone right-flag curves. If $\chi(\mathcal{F}) > 2^{5\lambda_k + 116}$, then $\mathcal{F}$ contains $k+1$ pairwise crossing members or a type 3 configuration.
\end{lemma}

\begin{proof}

Assume that $\mathcal{F}$ does not contain $k+1$ pairwise crossing members.  By Lemma \ref{distance}, for some $d \geq 2$, the subset of curves $\mathcal{F}^d$ at distance $d$ from $C_1$ satisfies

$$\chi(\mathcal{F}^d) \geq \frac{\chi(\mathcal{F})}{2} > 2^{5\lambda_k + 115}.$$

\noindent Recall that for each curve $C_i \in \mathcal{F}^d$, there is a path $C_1,C_{p_1},C_{p_2},...,C_{p_{d-1}},C_i$ such that $C_{p_t}$ is at distance $t$ from $C_1$. Now we define subsets $\mathcal{F}_1,\mathcal{F}_2 \subset \mathcal{F}^d$ as follows:

$$\mathcal{F}_1 = \{ C_i \in \mathcal{F}^d : \textnormal{there exists a path $C_1,C_{p_1},C_{p_2},...,C_{p_{d-1}},C_i$ with $p_{d-1} > i$.}\},$$

$$\mathcal{F}_2 = \{ C_i \in \mathcal{F}^d : \textnormal{there exists a path $C_1,C_{p_1},C_{p_2},...,C_{p_{d-1}},C_i$ with $p_{d-1} < i$.}\}.$$

\noindent Since $\mathcal{F}_1\cup \mathcal{F}_2 = \mathcal{F}^d$, either $\chi(\mathcal{F}_1) \geq \chi(\mathcal{F}^d)/2$ or $\chi(\mathcal{F}_2) \geq \chi(\mathcal{F}^d)/2$.  Since the following argument is the same for both cases, we will assume that

$$\chi(\mathcal{F}_1) \geq \frac{\chi(\mathcal{F}^d)}{2} \geq  2^{5\lambda_k +114}.$$

By Lemma \ref{sequence}, there exists a subset $\mathcal{H}_1 \subset \mathcal{F}_1$ such that $\chi(\mathcal{H}_1) >2$, and for every pair of curves $C_a,C_b\in \mathcal{H}_1$ that intersect, $\mathcal{F}_1(a,b)\geq 2^{5\lambda_k + 112}$.  Fix two such curves $C_a,C_b \in \mathcal{H}_1$ and let $\mathcal{A}$ be the set of curves in $\mathcal{F}(a,b)$ that intersects either $C_a$ or $C_b$.  By Lemma \ref{crucial}, there exists a subset $\mathcal{D}_1\subset \mathcal{F}_1(a,b)$ such that each curve $C_i \in \mathcal{D}_1$ is disjoint to $C_a$, $C_b$, $\mathcal{A}$, and moreover

 $$\chi(\mathcal{D}_1) \geq 2^{5\lambda_k + 112} -2^{\lambda_k + 1} -  k\cdot 2^{2\lambda_k + 102} > 2^{5\lambda_k + 111}.$$

\noindent Again by Lemma \ref{sequence}, there exists a subset $\mathcal{H}_2 \subset \mathcal{D}_1$ such that $\chi(\mathcal{H}_2) > 2^{\lambda_k}$, and for every pair of curves $C_u,C_v \in \mathcal{H}_2$ that intersect, $\chi(\mathcal{D}_1(u,v)) \geq 2^{4\lambda_k + 110}$.  Therefore, $\mathcal{H}_2$ contains $k$ pairwise crossing members $C_{i_1},C_{i_2}...,C_{i_k}$ for $i_1 < i_2 < \cdots < i_k$.  Since $\chi(\mathcal{D}_1(i_1,i_2)) \geq 2^{4\lambda_k + 110}$, by Lemma \ref{one}, there exists a subset $\mathcal{D}_2 \subset \mathcal{D}_1(i_1,i_2)$ such that

$$\chi(\mathcal{D}_2) > 2^{4\lambda_k + 110} - 2^{2\lambda_k} >  2^{4\lambda_k + 109},$$

\noindent and each curve $C_l \in \mathcal{D}_2$ is disjoint to the set of curves $\{C_{i_1},C_{i_2},...,C_{i_k}\}$.  Now we define a $2^{4\lambda_k + 107}$-sequence $\{r_i\}_{i = 0}^m$ of $\mathcal{D}_2$ such that $m \geq 4$.  Therefore, we have subsets

$$\mathcal{D}_2[r_0,r_1], \mathcal{D}_2(r_1,r_2],...,\mathcal{D}_2(r_{m-1},r_{m}]$$

\noindent such that

$$\chi(\mathcal{D}_2[r_0,r_1]) = \chi(\mathcal{D}_2(r_1,r_2]) =  2^{4\lambda_k + 107}.$$

\noindent By Lemma \ref{two}, we know that $\mathcal{D}_2[r_0,r_1]$ contains either a type 1 or type 2 configuration.
\medskip

Suppose that $\mathcal{D}_2[r_0,r_1]$ contains a type 2 configuration $\{\mathcal{K}_1,C_q\}$, where $\mathcal{K}_1$ is the set of $k$ pairwise intersecting curves.  See Figure \ref{last1}. $C_q \in \mathcal{D}_2[r_0,r_1] \subset \mathcal{F}^d$ implies that there exists a path $C_1,C_{p_1},C_{p_2},...,C_{p_{d-1}},C_{q}$ such that $p_{d-1}> q$.   Let $R$ be the region enclosed by the $y$-axis, $C_{a}$ and $C_{b}$.  Since $C_q$ lies inside of $R$, and $C_1$ lies outside of $R$, there must be a curve $C_{p_t}$ that intersects either $C_{a}$ or $C_{b}$ for some $1 \leq t \leq d-1$.  Since $C_a,C_b \in \mathcal{F}^d$, $C_{p_{d-1}}$ must be this curve.  Moreover, $C_q \in \mathcal{D}_1 \subset \mathcal{F}_1$ implies that $p_{d-1} > b$.  Since our curves are $x$-monotone and $x(C_q) \leq x(\mathcal{K}_1)$, $C_{p_{d-1}}$ intersects all of the curves in $\mathcal{K}_1$.  This creates $k+1$ pairwise crossing members in $\mathcal{F}$ and we have a contradiction.  See Figure \ref{last2}.

 \begin{figure}[h]
  \centering
\subfigure[Type 2 configuration.]{\label{last1}\includegraphics[width=0.39\textwidth]{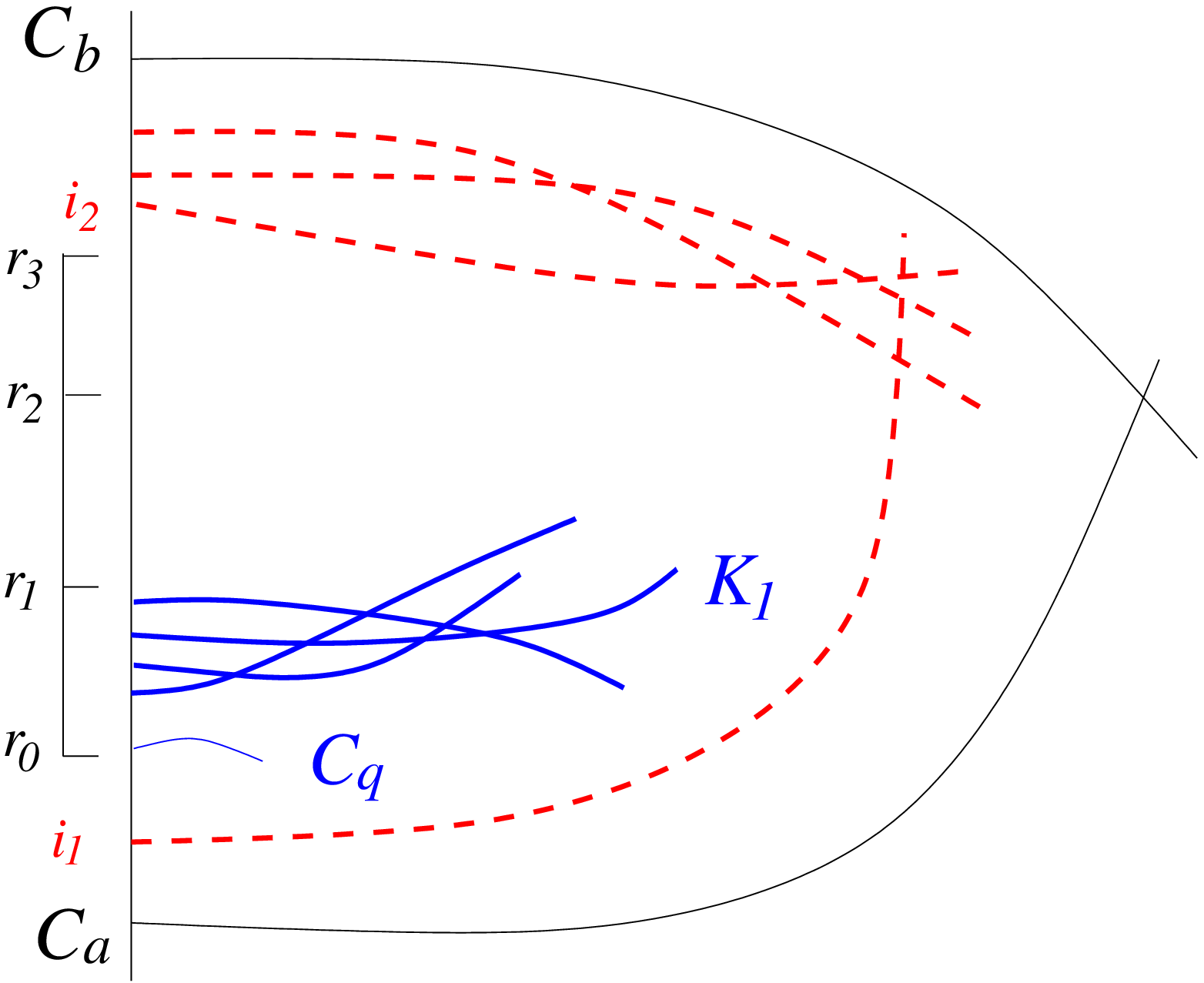}}\hspace{1cm}
\subfigure[$k+1$ pairwise crossing curves.]{\label{last2}\includegraphics[width=0.41\textwidth]{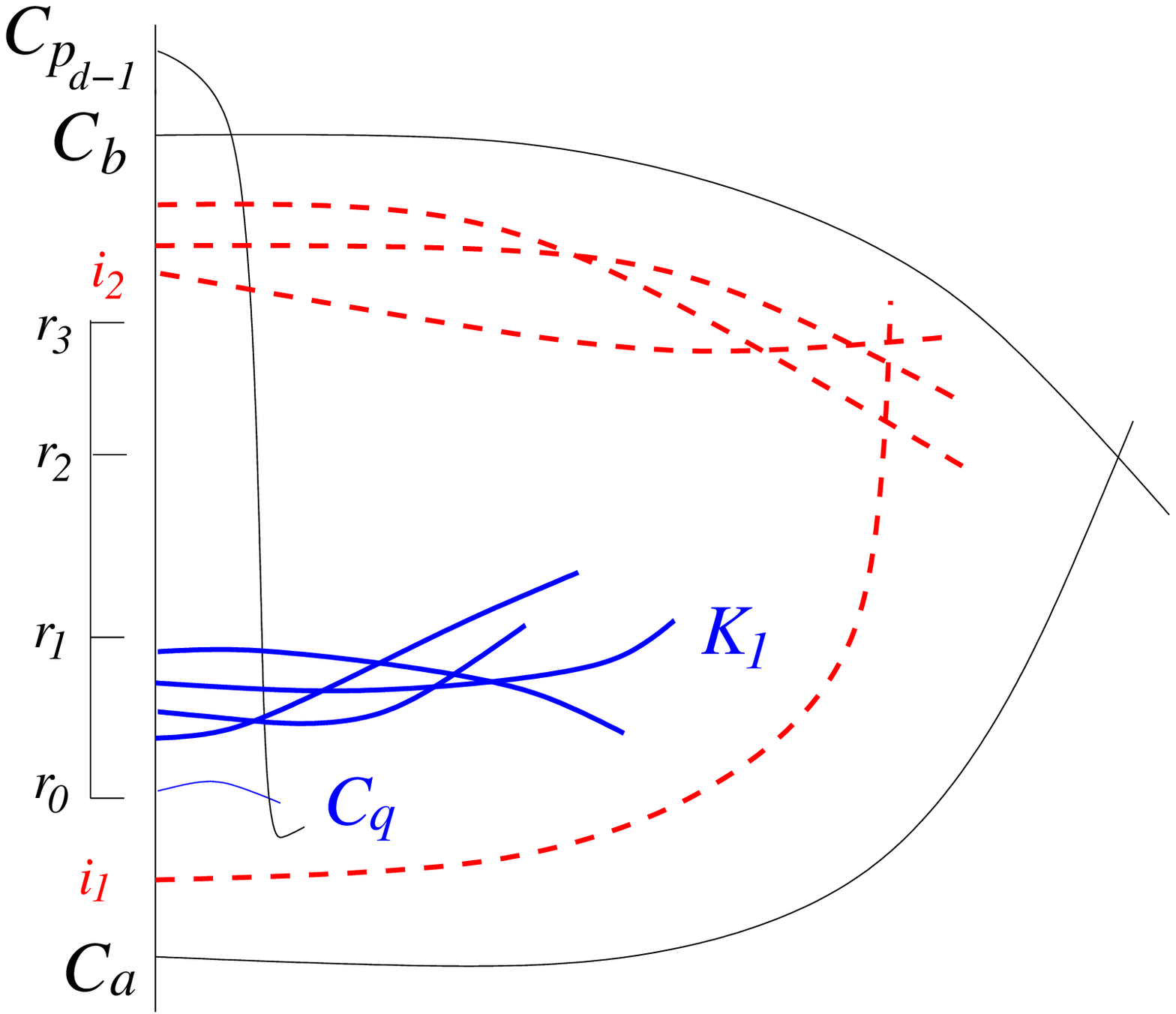}}
                        \caption{Lemma \ref{three}.}
\end{figure}

\medskip

Therefore we can assume that $\mathcal{D}_2[r_0,r_1]$ contains a type 1 configuration $\{\mathcal{K}_1,C_q\}$ where $\mathcal{K}_1$ is the set of $k$ pairwise intersecting curves.  See Figure \ref{last3}.  By the same argument as above, there exists a curve $C_{p_{d-1}}$ that intersects $C_q$ such that $p_{d-1} > b$.  Hence $\mathcal{K}_2=\{C_{i_2},C_{i_3},...,C_{i_k},C_{p_{d-1}}\}$ is a set of $k$ pairwise intersecting curves.  Since $\chi(D_2(r_1,r_2]) =  2^{4\lambda_k + 107}$, Lemma \ref{one} implies that there exists a curve $C_{q'} \in \mathcal{D}_1(r_1,r_2]$ that intersects neither $C_{q}$ nor $C_{p_{d-1}}$.  By construction of $\mathcal{D}_2(r_1,r_2]$ and by the definition of a type 1 configuration, $C_{q'}$ does not intersect any members in the set $\{\mathcal{K}_1,\mathcal{K}_2,C_q\}$.  By Lemma \ref{short}, we have

$$x(C_{q'})\leq x(C_q) \leq x(\mathcal{K}_1\cup \mathcal{K}_2),$$

\noindent and therefore $\mathcal{K}_1,\mathcal{K}_2,C_{q'}$ is a type 3 configuration.  See Figure \ref{last4}.

 \begin{figure}
  \centering
\subfigure[Type 1 configuration.]{\label{last3}\includegraphics[width=0.37\textwidth]{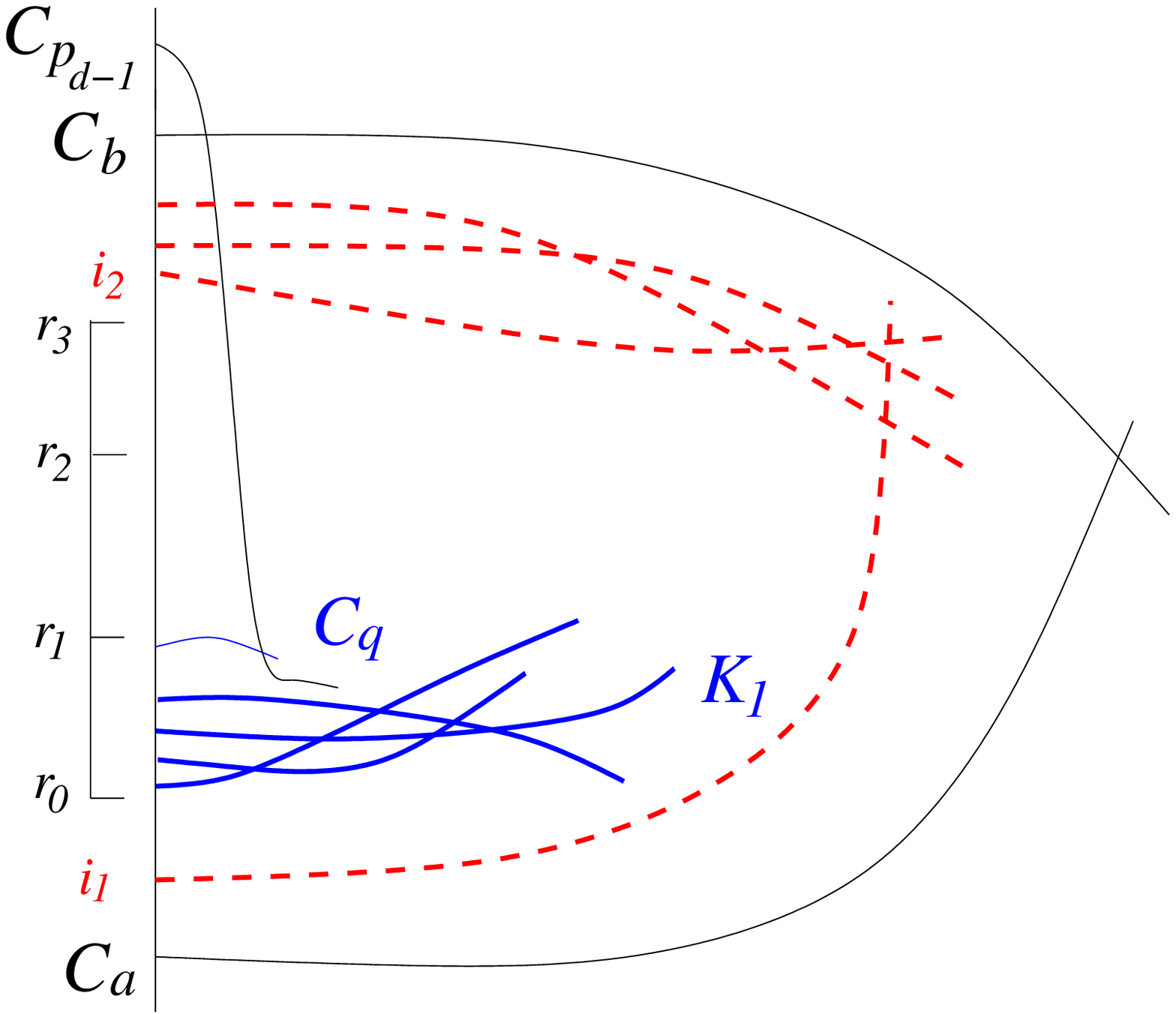}}\hspace{1cm}
\subfigure[Type 3 configuration.]{\label{last4}\includegraphics[width=0.39\textwidth]{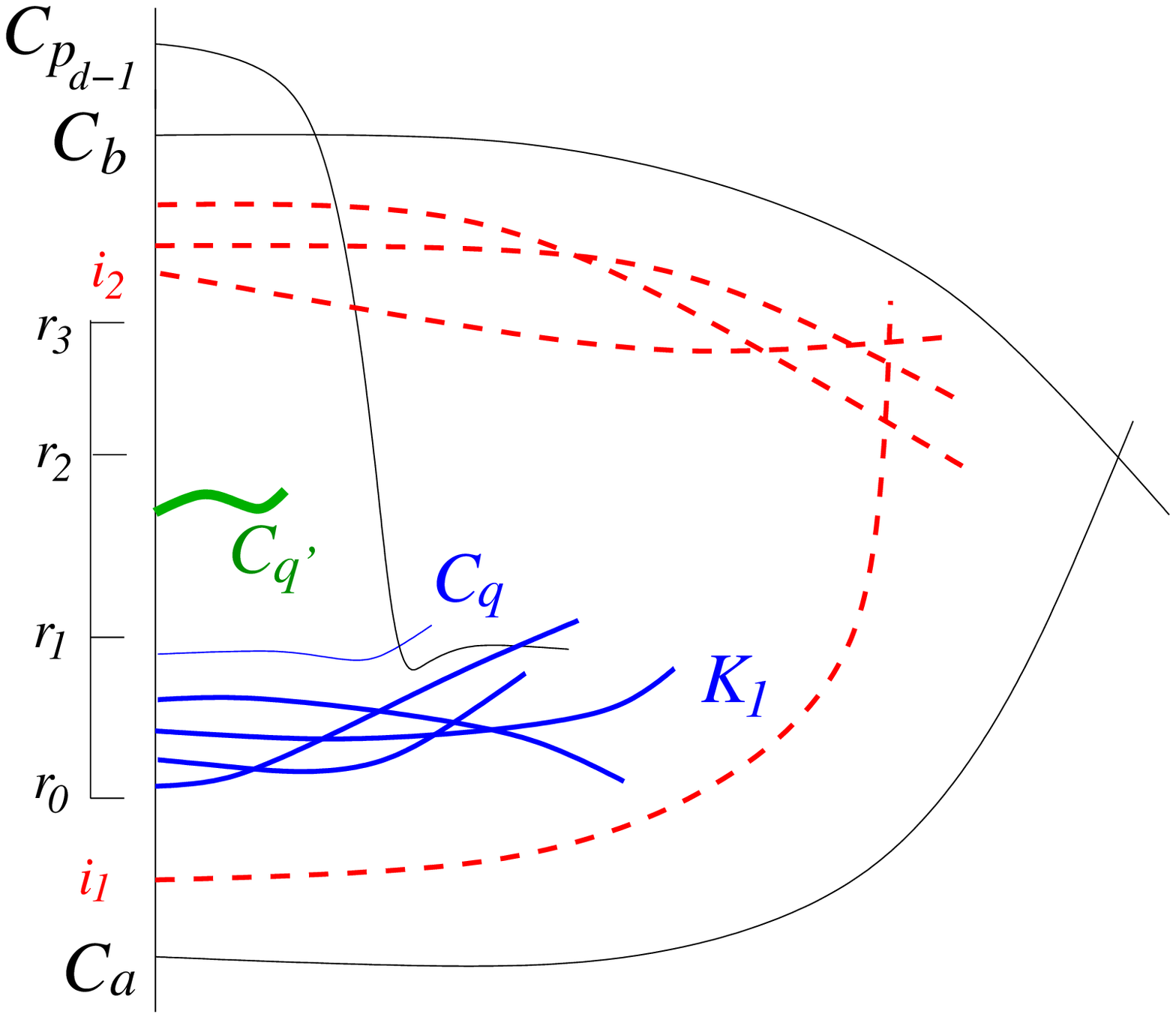}}
                        \caption{Lemma \ref{three}.}
  \label{last34}
  \end{figure}

\end{proof}

\section{Proof of the Theorem \ref{flag}}

The proof is by induction on $k$.  The base case $k = 2$ is trivial.  Now suppose that the statement is true up to $k$.   Let $\mathcal{F} = \{C_1,C_2,...,C_n\}$ be a simple family of $n$ $x$-monotone right-flag curves, such that $\chi(\mathcal{F}) > 2^{( 5^{k+2} - 121)/4}$.  We will show that $\mathcal{F}$ contains $k+1$ pairwise crossing members.  We define the recursive function $\lambda_k$ such that $\lambda_2 = 1$ and

$$\lambda_{k} = 5\lambda_{k-1} + 121 \hspace{1cm}\textnormal{for}\hspace{1cm}k \geq 3.$$

\noindent This implies that $\lambda_k =  ( 5^{k+1} - 121)/4$ for all $k \geq 2$.   Therefore, we have

$$\chi(\mathcal{F}) > 2^{( 5^{k+2} - 121)/4} = 2^{\lambda_{k + 1}}  = 2^{5\lambda_k + 121}.$$

 \noindent Just as before, there exists an integer $d \geq 2$, such that the set of curves $\mathcal{F}^d \subset \mathcal{F}$ at distance $d$ from the curve $C_1$ satisfies

$$\chi(\mathcal{F}^d) \geq \frac{\chi(\mathcal{F})}{2} > 2^{5\lambda_k + 120}.$$

\noindent Now we can assume that $\mathcal{F}^d$ does not contain $k+1$ pairwise crossing members, since otherwise we would be done.  By Lemma \ref{sequence}, there exists a subset $\mathcal{H} \subset \mathcal{F}^d$ such that $\chi(\mathcal{H}) > 2$, and for every pair of curves $C_a,C_b \in \mathcal{H}$ that intersect, $\mathcal{F}^d(a,b) \geq 2^{5\lambda_k + 118}$.  Fix two such curves $C_a,C_b \in \mathcal{H}$, and let $\mathcal{A}$ be the set of curves in $\mathcal{F}(a,b)$ that intersects $C_a$ or $C_b$.  By Lemma \ref{crucial}, there exists a subset $\mathcal{D}  \subset \mathcal{F}^d(a,b)$ such that each curve $C_i \in \mathcal{D}$ is disjoint to $C_a$, $C_b$, $\mathcal{A}$, and moreover

$$\chi(\mathcal{D}) \geq 2^{5\lambda_k + 118}  - 2^{\lambda_k + 1} - 2^{2\lambda_k + 102}  \geq 2^{5\lambda_k + 117}.$$

\noindent By Lemma \ref{three}, $\mathcal{D}$ contains a type 3 configuration $\{\mathcal{K}_1,\mathcal{K}_2,C_q\}$, where $\mathcal{K}_t$ is a set of $k$ pairwise intersecting curves for $t \in \{1,2\}$.  See Figure \ref{final}.  Just as argued in the proof of Lemma \ref{three}, since $C_q \in \mathcal{D} \subset \mathcal{F}^d$, there exists a path $C_1,C_{p_1},C_{p_2},...,C_{p_{d-1}},C_{q}$ in $\mathcal{F}$ such that either $p_{d-1} < a $ or $p_{d-1} > b$.  Since our curves are $x$-monotone and $x(C_q) \leq x(\mathcal{K}_1\cup \mathcal{K}_2)$, this implies that either $\mathcal{K}_1\cup C_{p_{d-1}}$ or $\mathcal{K}_2\cup C_{p_{d-1}}$ are $k+1$ pairwise crossing curves.  See Figures \ref{finala} and \ref{finalb}.

\begin{figure}[h]
\begin{center}
\includegraphics[width=160pt]{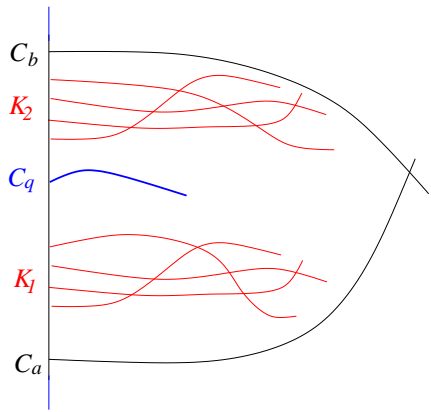}
  \caption{Type 3 configuration.}
  \label{final}
 \end{center}
\end{figure}

 \begin{figure}[h]
  \centering
\subfigure[$\mathcal{K}_1\cup C_{p_{d-1}}$ are $k+1$ pairwise crossing curves.]{\label{finala}\includegraphics[width=0.4\textwidth]{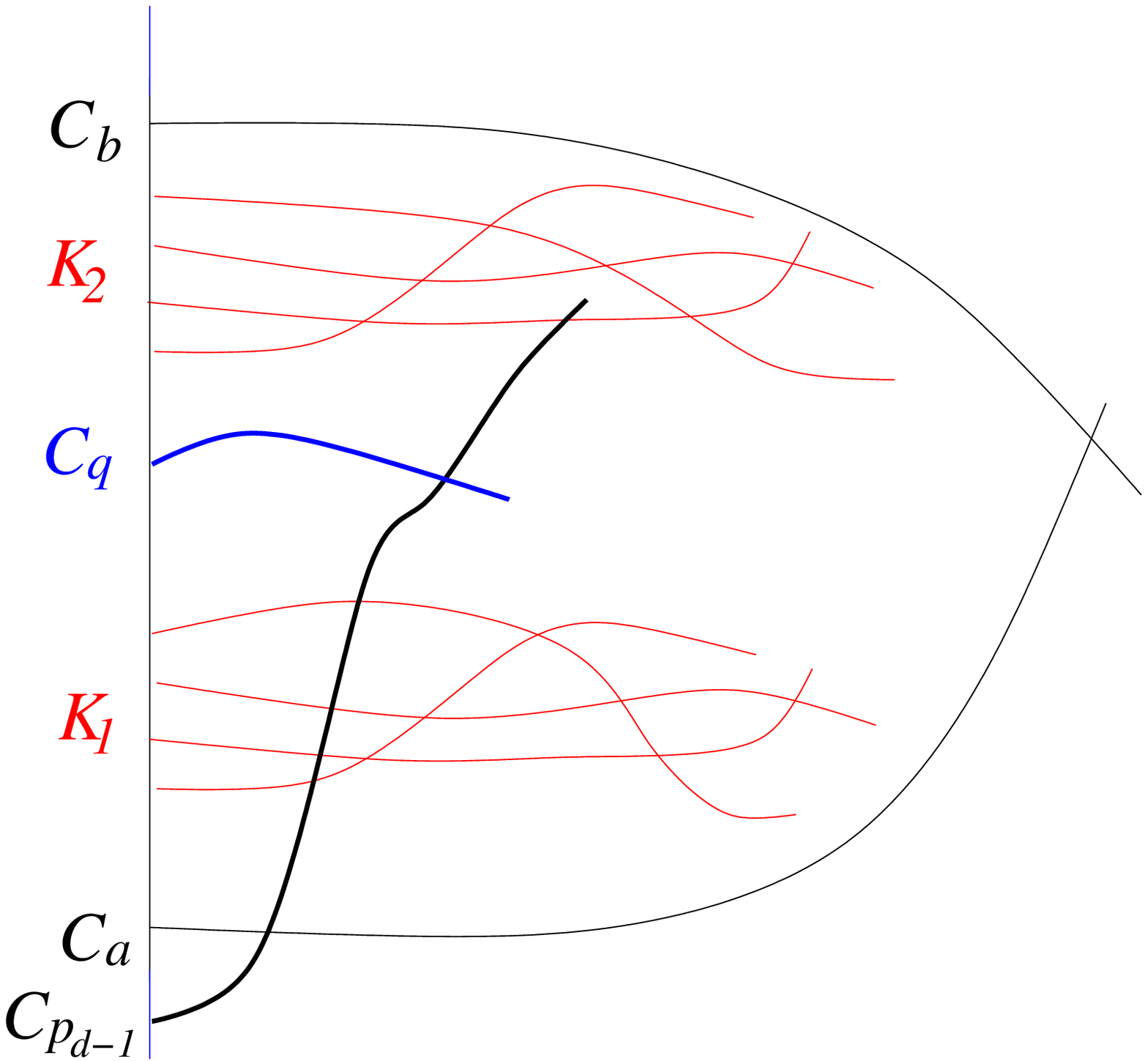}}\hspace{1cm}
\subfigure[$\mathcal{K}_2\cup C_{p_{d-1}}$ are $k+1$ pairwise crossing curves.]{\label{finalb}\includegraphics[width=0.4\textwidth]{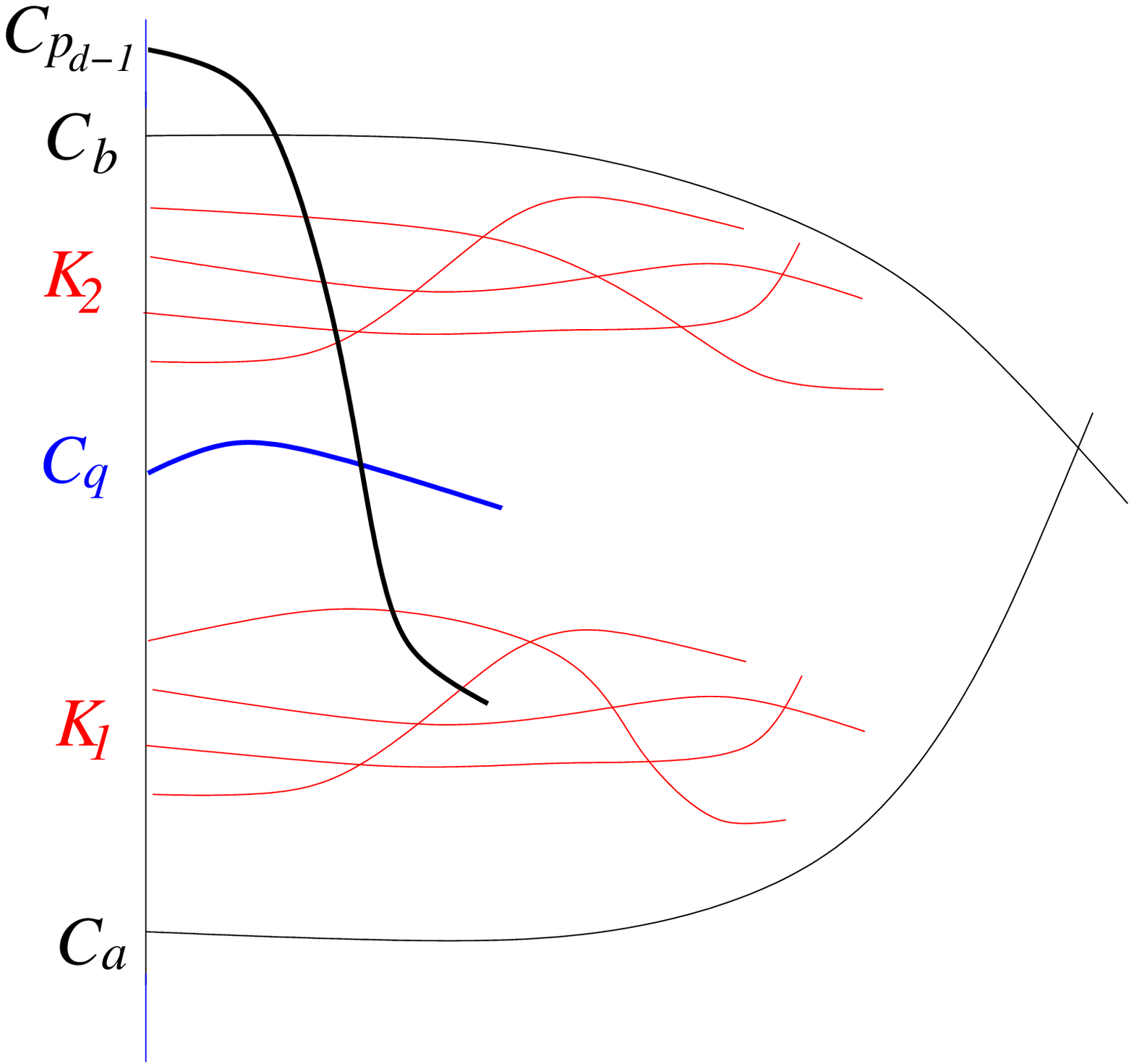}}
                        \caption{$k+1$ pairwise crossing curves.}
  \label{finalab}
\end{figure}

 $\hfill\square$

\end{document}